%% file: main.tex
\let\ORIlabel\label
\let\ORIrefstepcounter\refstepcounter
\AddToHook{package/hyperref/before}{%
  \let\label\ORIlabel
  \let\refstepcounter\ORIrefstepcounter
}
\documentclass[onefignum,onetabnum]{siamonline171218}

\usepackage{graphicx}
\usepackage[numbers]{natbib}
\setcitestyle{numbers,round}
\usepackage{subcaption}
\usepackage[utf8]{inputenc} 
\usepackage[T1]{fontenc}    
\usepackage{hyperref}       
\usepackage{url}            
\usepackage{booktabs}       
\usepackage{amsfonts}       
\usepackage{nicefrac}       
\usepackage{xcolor}   
\usepackage{graphicx}
\usepackage{makecell}
\usepackage{comment}
\usepackage{svg}

\usepackage{graphicx}
\usepackage{mathtools}
\usepackage{amssymb}
\usepackage[makeroom]{cancel}
\usepackage{tikz}

\usetikzlibrary{positioning,calc}
\usetikzlibrary{arrows.meta}
\usepackage{tikz-cd}

\usepackage{parskip}
\usepackage{setspace}
\usepackage{enumerate}

\usepackage[utf8]{inputenc}
\usepackage[T1]{fontenc}
\usepackage{hyperref}
\usepackage{url}
\usepackage{booktabs}
\usepackage{amsfonts}
\usepackage{nicefrac} 
\usepackage{xcolor}
\usepackage{array}

\renewcommand{\thefigure}{\arabic{figure}}
\newcolumntype{L}[1]{>{\raggedright\arraybackslash}p{#1}}

\ifpdf
\hypersetup{
  pdftitle={KANDy: Kolmogorov-Arnold Networks and Dynamical System Discovery},
  pdfauthor={K. Slote, J. Fish, and E. Bollt}
}
\fi

\input{ex_shared}

\begin{document}

\maketitle

\begin{abstract}
We introduce the Kolmogorov-Arnold Network for Dynamics (KANDy) as a zero-depth, wide neural architecture capable of discovering governing equations in chaotic and complex dynamical systems. Building on the foundation of Kolmogorov-Arnold Networks (KANs), KANDy explicitly learns governing equations by replacing sparse regression with a KAN. The synthesis of KANs and sparse regression addresses the limitations of equation discovery for KANs applied to dynamical systems and overcomes cases where sparse regression is hindered by sparsity constraints. Additionally, we show that our model, applied to the Hopf Fibration, recovers topological structure, thereby improving coherence with attractor properties. We apply our model to discrete and continuous dynamical systems, as well as to chaotic partial differential equations (PDEs). These results position KANDy as an interpretable and effective alternative for data-driven modeling of nonlinear dynamical systems. 
\end{abstract}

\section*{Introduction}

This work aims to synthesize two complementary approaches: Kolmogorov-Arnold Networks (KANs)~\cite{liu2025kan} and sparse regression for equation discovery in dynamical systems~\cite{sindy, rudy2017data}. In applied dynamical systems, a considerable obstacle is the rapid collapse of forecast accuracy in the presence of chaos. For example, when predicting the future state of a chaotic system such as the Lorenz attractor using advanced data-driven models, forecasted trajectories typically remain accurate for only one or two Lyapunov times before diverging significantly from the true system, resulting in a loss of predictive value. This rapid loss of fidelity is a routine obstacle that impedes meaningful forecasting and diminishes model utility in practical applications. Factors such as sensitivity to initial conditions, hidden symmetries, fractal attractor geometry, and violations of sparsity conditions collectively complicate both trajectory forecasting and the discovery of governing equations. Chaotic dynamical systems canonically exhibit exponential divergence of nearby trajectories and possess a fundamental predictability limit determined by the largest Lyapunov exponent~\cite{lorenz1963deterministic}. Even when governing equations are known, forecasts decorrelate on a timescale associated with the Lyapunov time, after which only statistical properties of the attractor are meaningful. When the equations are unknown, these limitations are further exacerbated. Successful data-driven models must infer dynamics from finite training windows that may span only a few Lyapunov times. Model discovery from observables is fundamentally ill-posed as an inverse problem and is further constrained by limited measurements~\cite{Bakarji2023DelayAutoencoders}. The proposed method is explicitly formulated to tackle these challenges by preserving the attractor geometry and the governing equation structure beyond conventional predictability limits, where other models fail.

Discovering models from observations and measurements of such systems is a fundamental problem for scientists and researchers. The data-driven identification, and forecasting of nonlinear dynamical systems has attracted a great deal of attention~\cite{FS1987,CM1987,Casdagli1989,SGMCPW1990,KR1990,GS1990,Gouesbet1991,TE1992,BBBB1992,Longtin1993,Murray1993,Sauer1994_PRL,Sugihara1994,FK1995,Parlitz1996,SSCBS1996,Szpiro1997,HKS1999,Bollt2000,HKMS2000,Sello2001,MNSSH2001,Smith2002,Judd2003,Sauer2004,YB2007,TZJ2007,WYLKG2011a,WLGY2011b,WYLKH2011c,SNWL2012a,SWL2012b,SWL2014a,SLWD2014b,SWFDL2014,SWWL2016,ASB2020} eliciting diverse methodologies e.g., calculating the information contained in sequential observations deducing deterministic equations~\cite{CM1987}, approximating a nonlinear system by linear equations~\cite{FS1987,Gouesbet1991,Sauer1994_PRL}, fitting differential equations~\cite{BBBB1992}, exploiting chaotic synchronization~\cite{Parlitz1996} or genetic algorithms~\cite{Szpiro1997,TZJ2007}, and the inverse Frobenius-Perron approach to designing a dynamical system ``near’’ the original system~\cite{Bollt2000}, or using the least-squares best approximation for modeling~\cite{YB2007}. Recently, a topic that has gained considerable interest is sparse optimization, where the assumption is that the system functions have a sparse structure, represented by a small number of elementary mathematical functions (e.g., a few power-series, Fourier-series, or polynomial terms), and the goal is to estimate the coefficients associated with these terms. A higher-order series expansion generally will have the vast majority of coefficients be identically zero, which is naturally formulated~\cite{WYLKG2011a,YLG2012} as a compressive-sensing problem~\cite{Candes2008Introduction, Baraniuk2007Compressed, Donoho2006Compressed, Candes2006Stable, Candes2006Robust}.

Sparse optimization is effective for systems where the governing equations are sufficiently sparse, such as the chaotic Lorenz~\cite{lorenz1963deterministic} and R\"{o}ssler~\cite{Rossler1976} systems, where the velocity fields contain a small number of low-order power-series terms. The sparsity condition presents a fundamental limitation: it works only when the system, and the equations in particular, have a sparse structure. Dynamical systems that violate the sparsity condition arise in biological and physical systems. For example, the Ikeda map, which describes the propagation of an optical pulse in a cavity with a nonlinear medium~\cite{Ikeda1979, HJM1985}, takes the form $x_{n+1} = 1 + u [x_n \cos(\theta_n) - y_n \sin(\theta_n)]$ and $y_{n+1} = u [x_n \sin(\theta_n) + y_n \cos(\theta_n)]$ with $\theta_n = k - p/(1 + x_n^2 + y_n^2)$, leading to an infinite series when expanded and thus violating the sparsity assumption~\cite{lai2021finding}. Additional discrete dynamical systems arise in ecological systems and gene-regulatory circuits whose governing equations have a Holling-type structure~\cite{Holling1959a,Holling1959b} that violates the sparsity condition~\cite{JHSLGHL2018}; for instance, the Holling Type II term $\frac{ax}{1 + bx}$ cannot be represented accurately with just a handful of polynomial terms. An example of a continuous dynamical system that lacks sparsity is the standard Kuramoto model, where each oscillator's dynamics are given by $\dot{\theta}_i = \omega_i + K/N \sum_{j=1}^{N} \sin(\theta_j - \theta_i)$, and the system identification degrades with network size and coupling complexity~\cite{Kuramoto1975SelfEntrapment, Basiri2025SINDyG}. Additional potential issues with sparse optimization include topological invariance, in which global topology cannot be simplified by coordinate transformations, truncation, or local learning. The Kuramoto-Sivashinsky (KS) equation and Inviscid Burgers’ equations are two such systems whose attractors possess translational and reflectional invariance, which can obstruct local solutions.

Forecasting chaotic systems, not only for Lorenz-like systems, but for chaotic attractors with spatiotemporal chaos such as the KS equation, several approaches stand out in terms of long-term forecasting, such as (i) Reservoir computing (RC), and in particular echo state networks (ESNs)~\cite{jaeger2001esn}, achieving autoregressive rollout forecasts greater than 6-7 Lyapunov times, (ii) continuous-time neural architectures, such as NeuralODEs~\cite{chen2018neuralode}, learning representative vector fields and smooth dynamics from trajectory data, and (iii) operator learning, e.g. DeepONet and other variants, by learning mappings between function spaces, enabling data-driven surrogates for parametric partial differential equations (PDEs) and turbulent flows~\cite{lu2021deeponet,li2020fourierOperator} as well as reproducing long-term ``climate’’ statistics~\cite{pathak2017lyapunov,pathak2018reservoir}. Across these families of models, a common observation is that, although forecast models’ short-horizon one-step errors are often minimal, autoregressive rollouts still exhibit exponential error growth and phase drift that co-occur with the system’s Lyapunov spectrum. Forecasting the Lorenz system is a canonical benchmark for data-driven chaotic dynamics modeling \cite{pathak2018reservoir, massaroli2020stabilityODE, rubanova2019latentODE,vlachas2018lstm,abarbanel1993lorenz,floriankrach2022NJODE,lu2020dataDrivenPDE}. The RC, neural-ODEs, and LSTM architectures have all demonstrated high short-term accuracy but fast error growth and phase decorrelation as prediction horizons reach the Lyapunov time \cite{vlachas2018lstm, pathak2018reservoir, massaroli2020stabilityODE}. Physics-informed neural networks and operator-learning methods have further improved geometric attractor reconstruction and interpretability for long-term rollouts, but still encounter fundamentally limited deterministic predictability due to chaos \cite{lu2020dataDrivenPDE, floriankrach2022NJODE}. An additional body of work on forecasting dynamical systems exploits Deep Learning (DL) approaches for chaotic systems, achieving accurate short-term forecasts but suffering from rapid error growth and phase decorrelation at longer horizons~\cite{vlachas2018lstm}. 

With the advent of continuous-time deep learning, neural ODEs~\cite{chen2018neuralode} and their variants~\cite{rubanova2019latentODE,kidger2020neuralCDE} have provided more flexible representations of dynamical systems. When trained on chaotic trajectories, such models can recover smooth vector fields consistent with the underlying ODEs, but their long-term rollouts still degrade due to accumulated phase errors. Recently, physics-informed neural networks (PINNs) and operator-learning approaches~\cite{lu2021deeponet,li2020fourierOperator} have attempted to learn differential operators directly, yet they generally require knowledge of governing equations or dense temporal supervision.

\begin{figure}[t]
\centering
\includegraphics[width=1.0\linewidth]{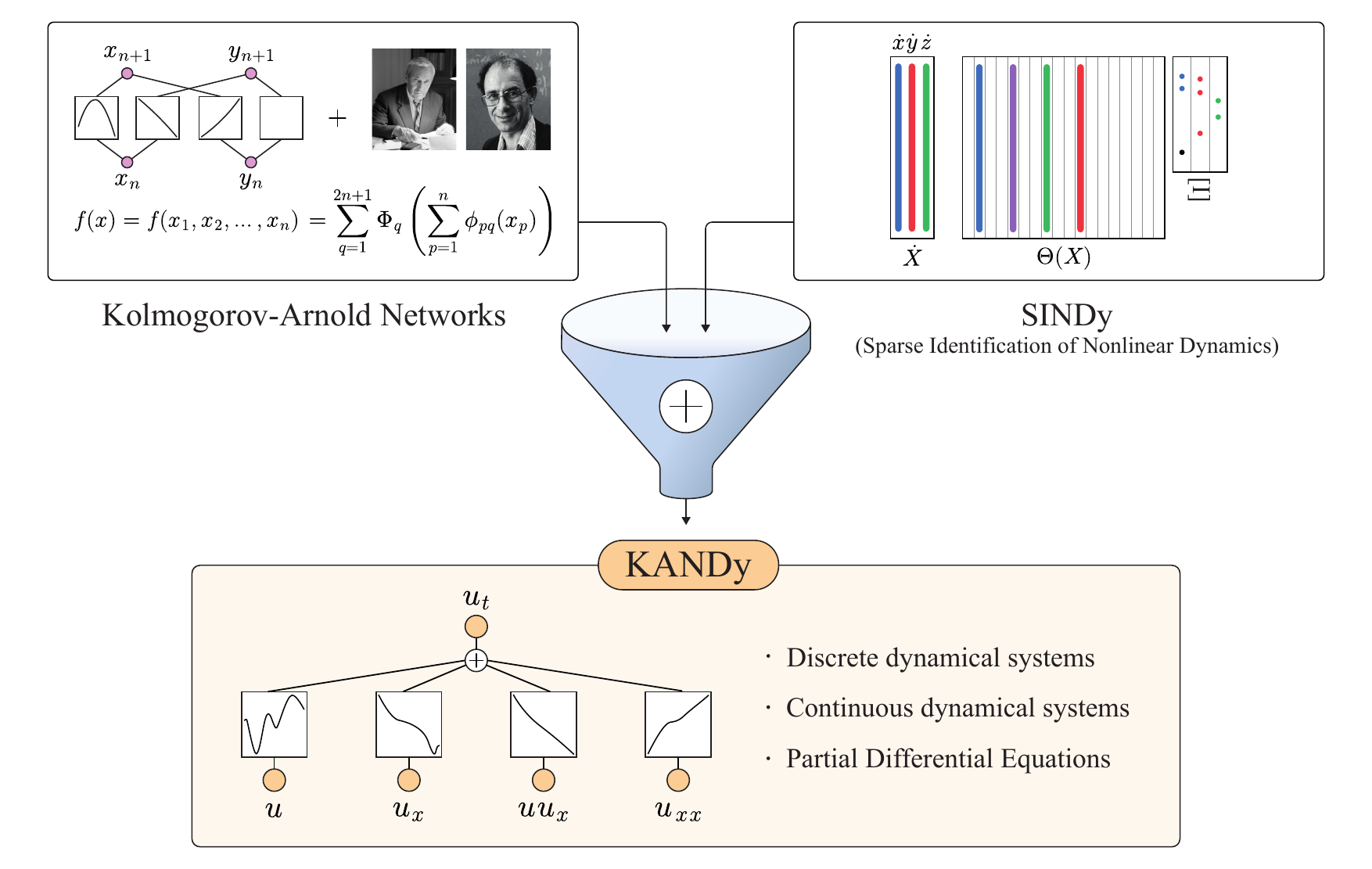}
\caption{Left: Kolmogorov–Arnold Networks (KANs), inspired by the Kolmogorov–Arnold representation theorem (KAT; portraits of Andrey Kolmogorov and Vladimir Arnold), which states that multivariate functions can be expressed as sums of compositions of univariate functions. KAN replaces traditional neural network weights with learnable spline functions. Right: Sparse Identification of Nonlinear Dynamics (SINDy), where measurements of system states form a library of candidate nonlinear terms and sparse regression selects a minimal set explaining the dynamics. Bottom: The proposed KANDy framework maps system states to a lifted invariant subspace whose features act as a candidate library. A wide, zero-depth KAN then learns nonlinear transformations approximating the system vector field, enabling interpretable symbolic governing equations while relaxing strict library selection and sparsity requirements. The core advantage of this approach is that it is forgiving of a priori library selection and relaxes the sparsity condition required for sparse regression. 
}
\label{fig:flow-chart}
\end{figure}

At the same time, recent work on Kolmogorov-Arnold Networks (KANs) introduces new classes of architectures (Figure~\ref{fig:flow-chart}) replacing conventional weights and biases with learnable univariate spline functions inspired by the Kolmogorov-Arnold representation theorem~\cite{liu2025kan, liu2024kan2}, which functions as a dimension reduction applied to the inner variables of continuous functions~\cite{Kolmogorov1957, Arnold2009_Kolmogorov}. KANs exhibit marked advantages over standard neural networks for applications to dynamical systems, including interpretability and the ability to capture structured nonlinearities. By parameterizing nonlinear continuous functions using a small number of univariate components, KANs offer a natural framework for representing smooth vector fields, embeddings, and the latent coordinates of dynamical systems. This makes them attractive for chaos-aware sequence modeling, where both local nonlinear transformations and global geometric structure are essential.

KANs were introduced as interpretable and efficient alternatives to multilayer perceptrons (MLPs), leveraging learnable activation functions based on the Kolmogorov–Arnold representation theorem \cite{liu2025kan,coxkan2025, oxfordgenomicskan2025,xkan2024,naturegraphkan2025,scidirectkan2025,iclrkan2024}. Recently, it has been shown that much smaller KAN architectures can match or outperform standard MLPs in regression, PDE solving, and genomics, often with lower computational cost and greater transparency. For example, CoxKAN applies KANs to survival analysis, yielding more interpretable symbolic hazard function formulas and outperforming both classical statistical models and deep learning alternatives \cite{coxkan2025}. Similarly, LKANs and CKANs in genomics tasks achieve results comparable to or better than dense and convolutional baselines, with enhanced efficiency and feature selection capabilities \cite{oxfordgenomicskan2025}. A key aspect of this interpretability is that KANs often yield explicit symbolic expressions for learned mappings. For instance, consider a toy regression where a KAN with two univariate spline functions learns to approximate a target such as $f(x, y) = 0.5 \sin(x) + 2y$: KAN can directly represent this as $\phi_1(x) + \phi_2(y)$, with $\phi_1(x) \approx 0.5 \sin(x)$ and $\phi_2(y) \approx 2y$, while for a compositional function one may have $f(x, y) = xy = \exp{(\log(x) + \log(y)})$, making the learned model transparent and human-readable. This level of explicitness highlights the transparency advantage of KANs over traditional deep neural networks.

Algorithmic advances such as X-KAN introduce evolutionary rule-based training and demonstrate that local KAN models outperform global MLP and KAN baselines for highly nonlinear targets \cite{xkan2024}. Recent theoretical and experimental results confirm that KANs offer superior scaling laws and direct physical interpretability, supporting their adoption in physics, symbolic regression, molecular property inference, and scientific machine learning \cite{liu2025kan,naturegraphkan2025,scidirectkan2025,iclrkan2024}.

This work introduces a new architecture, the Kolmogorov-Arnold Network for Dynamics (KANDy), characterized by zero deep layers — equivalent to a regression — and increased-width parameters that incorporate physically informed, or \emph{lifted}, terms. This design bridges the gap between sparse regression and KANs, enabling the discovery of governing equations from dynamical systems.

The proposed architecture is evaluated on canonical chaotic benchmarks of increasing complexity. The evaluation focuses on three primary criteria: (i) equation discovery accuracy, assessing how closely the recovered functional form matches the true governing equations; (ii) rollout stability and attractor geometry preservation, measuring the consistency and long-term fidelity of predictions relative to the true system; and (iii) symmetry and topology recovery, evaluating the model's ability to capture invariant structural or geometric properties of the underlying dynamics. For low-dimensional chaos derived from ordinary differential equations (ODEs), the Lorenz system is used as an example, given its well-characterized strange attractor and Lyapunov spectrum~\cite{lorenz1963deterministic}. To demonstrate KANDy's effectiveness in high-dimensional spatiotemporal chaos, the KS-equation with periodic boundary conditions, and Inviscid-Burgers' equation with 20 random Fourier modes are employed, representing a prototypical partial differential equation (PDE) that exhibits pattern formation, broadband spectra, and extensive Lyapunov spectra~\cite{hyman1986kseqn}. In both scenarios, a zero-depth KAN is trained on library-derived inputs from system states, and multi-step forecasting is performed using a correlative and spectral loss that prioritizes near-term accuracy while maintaining long-horizon structural coherence.

\subsection*{Recent Work}

In \citet{Panahi2025KANModelDiscovery}, it is shown that KANs applied to discrete dynamical systems reconstruct the statistical properties of attractors, but the authors provide no universal principled way to estimate governing equations. Several works exist that apply KANs to ODEs and dynamical systems. In \citet{Bagrow2025}, Multi-exit Kolmogorov–Arnold networks, the learning-to-exit algorithm supplies a novel loss function to improve the prediction of KANs applied to dynamical systems. In \citet{Koenig2024}, the LEAN-KAN algorithm applies symbolic regression to KAN activation weights to recover the Lotka–Volterra predator–prey model, and this work continues Kan-ODE, and in \citet{Koenig2025}, successive work moves to improve the estimation of equations by improving the Multiplication layer of KANs to improve estimation.

\subsection*{Our Contribution}

The results demonstrate that KANDy, specifically zero-depth and wide KANs with physically-informed or lifted features, (i) discovers governing equations of dynamical systems, (ii) provides enhanced modeling capability where sparse regression is ineffective, (iii) maintains attractor geometry and stable autoregressive rollouts in cases where additional network depth either degrades or does not improve performance, and (iv) successfully discovers governing equations for multiple types of chaotic systems, serving as a benchmark where sparse regression methods struggle, even reconstructing the underlying network as a notable outcome. 
Qualitatively, it improves upon traditional sparse regression techniques by integrating learned model trajectories during training, rather than relying solely on least-squares solutions.

Additionally, our analysis highlights that deep-KANs alone are insufficient for estimating governing equations of dynamical systems since the dynamics of non-linear terms, such as the coupling term $xy$ of $\dot{z}$ in the Lorenz system cannot be reduced to 1D splines because the dynamics of this non-linear term live in a higher-dimensional manifold and a 1D reduction cannot exist within standard multiplication nodes. Additionally, typical KAN function fitting applies a function per edge and composes them, obfuscating the true global equation.

\section*{Mathematical Background}
\stepcounter{section}
\setcounter{equation}{0}

First, we introduce the mathematical formalism for the KAN\-Dy method, e.g. lifting coordinates of our systems to a linearized and infinite-dimensional space and the Kolmogorov-Arnold Representation theorem (KAT). 
We also introduce discrete and continuous dynamical systems, including the Henon and Lorenz systems, as well as chaotic PDE benchmarks such as the KS equation and the Inviscid Burgers equation.
Additionally, we introduce the Hopf fibration, in which the $3$-sphere, a three-dimensional surface sitting inside four-dimensional Euclidean space, can be viewed as built from circles arranged over the $2$-sphere, a two-dimensional surface sitting inside three-dimensional Euclidean space.

\subsection*{Kolmogorov-Arnold Representation theorem}

Vladimir Arnold and Andrey Kolmogorov established that if $f$ is a multivariate continuous function on a bounded domain, then $f$ can be written as a finite composition of continuous functions of a single variable and addition, e.g, for a continuous function $f:[0,1]^n\to\mathbb{R}$,
\begin{equation}\label{eq:KART}
    f(\mathbf{x}) = f(x_1,\cdots,x_n)=\sum_{q=1}^{2n+1} \Phi_q\left(\sum_{p=1}^n\phi_{q,p}(x_p)\right),
\end{equation}
where $\phi_{q,p}:[0,1]\to\mathbb{R}$ and $\Phi_q:\mathbb{R}\to\mathbb{R}$. The Kolmogorov-Arnold representation theorem (KAT)~\eqref{eq:KART} is itself a type of dimension reduction for multivariate functions, reducing a multivariate function into sums of univariate functions composed by outer continuous functions. However, constructive proofs of KAT are often heavily conditioned, and the inner functions are often ``Cantor-like," non-smooth, fractal, so they may not be learnable in practice~\cite{poggio2020theoretical,girosi1989representation}. Because of this pathological behavior, KAT was largely deprioritized in machine learning research, regarded as theoretically justified but useless in practice~\cite{poggio2020theoretical,girosi1989representation}. 

Following from Eq.~\eqref{eq:KART}, when envisioning this theorem applied to neural network architectures, the proposed architecture has only two non-linear layers and a small number of terms ($2n+1$) in the hidden layer, \citet{liu2025kan} generalizes this result to a network with arbitrary widths and depths observing that most functions in physics are smooth with sparse compositional structures facilitating Kolmogorov-Arnold representations ~\cite{lin2017does}.

Previous studies investigating the use of KAT to build neural networks~\cite{sprecher2002space,koppen2002training,lin1993realization,lai2021kolmogorov,leni2013kolmogorov,montanelli2020error, he2023optimal} mostly relied on the original depth-2, width-($2n+1$) representation, and many did not have the opportunity to leverage more modern techniques like  backpropagation to train the networks. In \citet{lai2021kolmogorov} used a depth-2 width-($2n+1$) representation, breaking the curse of dimensionality observed, both empirically and with an approximation theory. \citet{liu2025kan} generalizes the original KAT to arbitrary widths and depths, introducing KANS and highlighting its accuracy and interpretability.

\subsection*{Lorenz System}

The Lorenz system provides a canonical system for chaotic dynamical systems and for estimating governing equations. Its dynamics are governed by

\begin{equation}\label{eq:lorenz}
\begin{aligned}
\frac{dx}{dt} &= \sigma (y - x), \\
\frac{dy}{dt} &= x(\rho - z) - y, \\
\frac{dz}{dt} &= xy - \beta z,
\end{aligned}
\end{equation}

\noindent
where $\sigma$, $\rho$, and $\beta$ denote the system parameters. The solution is given by the state variables $x(t)$, $y(t)$, and $z(t)$ representing idealized components of the system: $x(t)$ is proportional to the intensity of convective motion, $y(t)$ represents the horizontal temperature difference between ascending and descending currents, and $z(t)$ corresponds to the vertical temperature stratification. The parameters $\sigma > 0$, $\rho > 0$, and $\beta > 0$ are constant system parameters, where  $\sigma$ is the coupling strength between temperature and velocity fields, $\rho$ governs the intensity of thermal forcing, and $\beta$ is a geometric factor related to the vertical temperatures. Different choices of these parameters lead to qualitatively distinct dynamical regimes, including steady states, periodic orbits, and chaotic behavior. This system exhibits a strange attractor characterized by sensitivity to initial conditions.

\subsection*{Kuramoto--Sivashinsky equation}
The KS equation serves as a standard benchmark for spatial-temporal systems. A pseudospectral scheme is used to numerically integrate the system's nonlinear dynamics. In contrast to low-dimensional chaotic systems, such as the Lorenz attractor, the KS equation exhibits broadband temporal spectra, long-range spatial coupling, and an extensive Lyapunov spectrum. These properties make it an ideal test case for evaluating the proposed KANDy architecture's ability to learn complex PDE-driven chaotic behavior and governing equations.

\begin{equation}\label{eq:ks_governing}
\frac{\partial u}{\partial t}
= -u \frac{\partial u}{\partial x}
 - \nu \frac{\partial^2 u}{\partial x^2}
 - \frac{\partial^4 u}{\partial x^4},
\end{equation}

The KS equation on a periodic domain $x \in [0,L]$ is given by Equation~\eqref{eq:ks_governing} where $u(x,t)$ is the dynamical field and $\nu > 0$ is the effective viscosity.

\subsection*{Inviscid Burgers}

The inviscid Burgers equation is a system of PDEs describing nonlinear wave propagation and shock formation in hyperbolic conservation laws, where our experiment shows that it is an edge case where sparse regression struggles. Its dynamics are governed by

\begin{equation}\label{eq:inciscid_burgers_governing}
\frac{\partial u}{\partial t} + u \frac{\partial u}{\partial x} = 0,
\end{equation}

\noindent
where $u(x,t)$ denotes a scalar velocity field defined over a one-dimensional spatial domain. The system state at time $t$ is given by

$$
s(t) = u(\cdot,t).
$$

Despite its relatively simple form, the inviscid Burgers equation exhibits nonlinear behavior, including the finite-time development of discontinuities – in the form of shock waves -- from smooth initial conditions. Here, $u(x,t)$ represents the local transport velocity, and the nonlinear advective term $u \partial_x u$ induces steepening of wave fronts. In the absence of viscosity, characteristics may intersect, leading to multivalued solutions and singularities, necessitating the introduction of weak solutions. Depending on the choice of initial condition, the system exhibits smooth evolution over finite time horizons or rapid shock formation, suggesting that sparse regression may struggle.

\begin{equation}\label{eq:inciscid_burgers_governing_viscocity}
\frac{\partial u}{\partial t} + u \frac{\partial u}{\partial x} = \nu \frac{\partial^2 u}{\partial x^2},
\end{equation}

Incorporating a viscosity term $\nu\frac{\partial^2 u}{\partial x^2}$ as in Eq.\eqref{eq:inciscid_burgers_governing_viscocity}, the wave fronts clash and steepen, developing shocks which later dissipate. The wave fronts dissipate at a scale that is dependent on the initial conditions.

\subsection*{Toy Example: Hopf-Fibration}

\begin{figure}[!htp]
    \centering
\includegraphics[width=\textwidth]{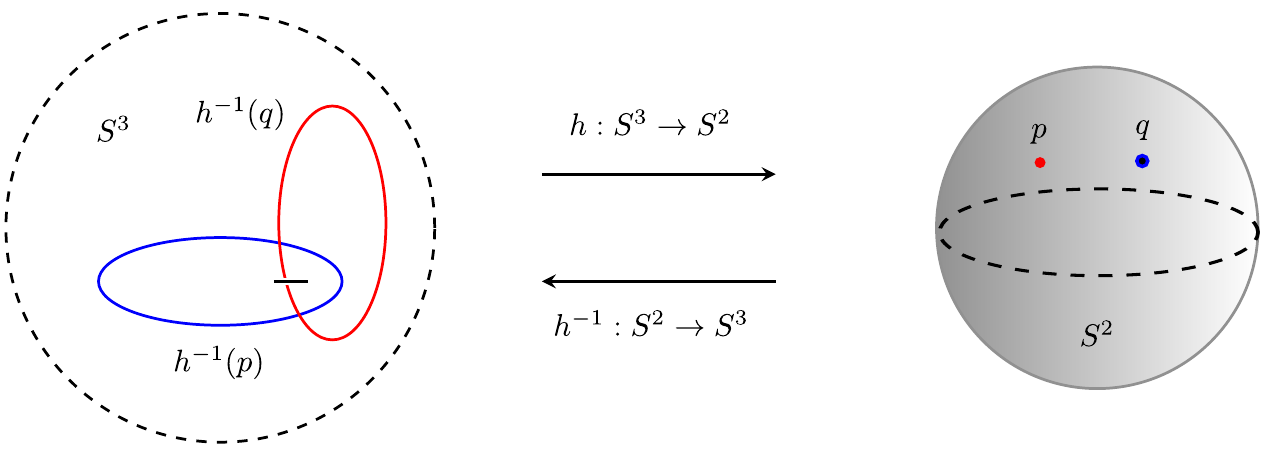}
\caption{
Left: The $3$-sphere $S^3$, which is a three-dimensional manifold sitting naturally in $\mathbb{R}^4$, with two distinct Hopf fibers $h^{-1}(p)$ and $h^{-1}(q)$, shown as linked circles inside a dashed boundary representing $S^3$. Any two distinct fibers of the Hopf fibration are linked once, illustrating the nontrivial topology of the fibration. Right: The $2$-sphere $S^2$, which is a two-dimensional manifold sitting naturally in $\mathbb{R}^3$ as the base space, with two points $p$ and $q$. Under the Hopf map $h: S^3 \to S^2$, each point in $S^2$ corresponds to a circle (fiber) in $S^3$.
    }
    \label{fig:hopf_schematic}
\end{figure}

To systematically evaluate KANDy’s ability to discover invariant structure in complex systems, we select the Hopf fibration as a representative topological example that serves as a potential pathological case for standard KANs. These cases were chosen because they each embody a distinctive class of mathematical invariance—quotient topology, fractal hierarchy, and extreme irregularity—offering a comprehensive testbed for our method. We show that KANDy can both identify and represent these invariances from data, demonstrating expressive power across diverse geometries and symmetries.

Figure~\ref{fig:hopf_schematic} illustrates the Hopf fibration, which encodes invariant structure arising from the quotient topology. To build geometric intuition, visualize $S^3$ as the set of points in four-dimensional space lying at a fixed distance from the origin. Imagine that every point on the familiar two-dimensional sphere $S^2$ can be associated with a continuous family of circles, each wrapping around in an unexpected way, such that every circle is linked with every other—a structure impossible to reproduce in ordinary three-dimensional space. The Hopf fibration organizes these circles so that $S^3$ is woven from interlocked loops, with each loop projecting down to a single point on $S^2$.

The Hopf fibration is a canonical example of a nontrivial fiber bundle $h : S^{3} \longrightarrow S^{2},$ in which the total space $S^{3}\subset \mathbb{C}^{2}$ fibers over the base space $S^{2}\subset \mathbb{R}^{3}$ with fiber $S^{1}$. Writing a point of $S^{3}$ as a pair of complex numbers $(z_{1},z_{2})\in\mathbb{C}^{2}$ satisfying $|z_{1}|^{2}+|z_{2}|^{2}=1$, the Hopf map is defined explicitly by
$$
h(z_{1},z_{2})
=
\bigl(
2\,\mathrm{Re}(z_{1}\overline{z}_{2}),
\;
2\,\mathrm{Im}(z_{1}\overline{z}_{2}),
\;
|z_{1}|^{2}-|z_{2}|^{2}
\bigr),
$$
which indeed satisfies $x^{2}+y^{2}+z^{2}=1$, hence $h(z_{1},z_{2})\in S^{2}$. The map $h$ is invariant under the free $S^{1}$-action
$$
(z_{1},z_{2}) \;\mapsto\; (e^{i\theta}z_{1},\,e^{i\theta}z_{2}),
\qquad \theta\in[0,2\pi),
$$
so that each fiber is an orbit of this action and is diffeomorphic to a circle:
$$
h^{-1}(p)\cong S^{1}\quad \text{for all } p\in S^{2}.
$$
Thus $S^{3}$ realizes a principal $S^{1}$-bundle over $S^{2}$, which is topologically nontrivial and cannot be written as a global product $S^{2}\times S^{1}$.

A key geometric feature of the Hopf fibration is that distinct fibers are linked in $S^{3}$: for any two distinct points $p,q\in S^{2}$, the preimages $h^{-1}(p)$ and $h^{-1}(q)$ form linked circles with linking number one. This linking encodes the nontrivial topology of the bundle and reflects the fact that $S^{2}$ arises as the quotient
$$
S^{2} \;\cong\; S^{3}/S^{1}.
$$
From the perspective of symmetry reduction, the Hopf fibration provides a geometric mechanism by which invariant structure is transferred from the quotient space $S^{2}$ back to the full space $S^{3}$. Functions or dynamics on $S^{3}$ that are invariant under the $S^{1}$-action factor through \(h\), i.e.,
$$
f(z_{1},z_{2}) = \tilde f\bigl(h(z_{1},z_{2})\bigr),
$$
for some function $\tilde f : S^{2}\to\mathbb{R}$. This decomposition illustrates how invariant coordinates arise naturally from quotient topology, providing a geometric foundation for discovering low-dimensional invariants in symmetric dynamical systems.

\section*{Methods}
\stepcounter{section}
\setcounter{equation}{0}

Next, we formalize the KANDy framework for learning governing equations of nonlinear dynamical systems using lifted zero-depth KANs. We first describe the architecture and lifting mechanism that augments state variables with physics-informed or ''lifted" nonlinear features. We then introduce the combined derivative and rollout training objective, which enforces both local vector field accuracy and global trajectory consistency. Then, we present the symbolic complexity-regularized edge selection procedure used to recover interpretable governing equations. Finally, we outline the chaos diagnostics—Lyapunov time, normalized rollout error, and correlation dimension—used to evaluate our results and analysis. Together, these components define a unified framework for equation discovery and long-horizon forecasting in chaotic systems.

\subsection*{Zero-depth Architecture, Learning Mechanism, Lifted Features}






\begin{figure}
    \centering
\includegraphics[width=\linewidth]{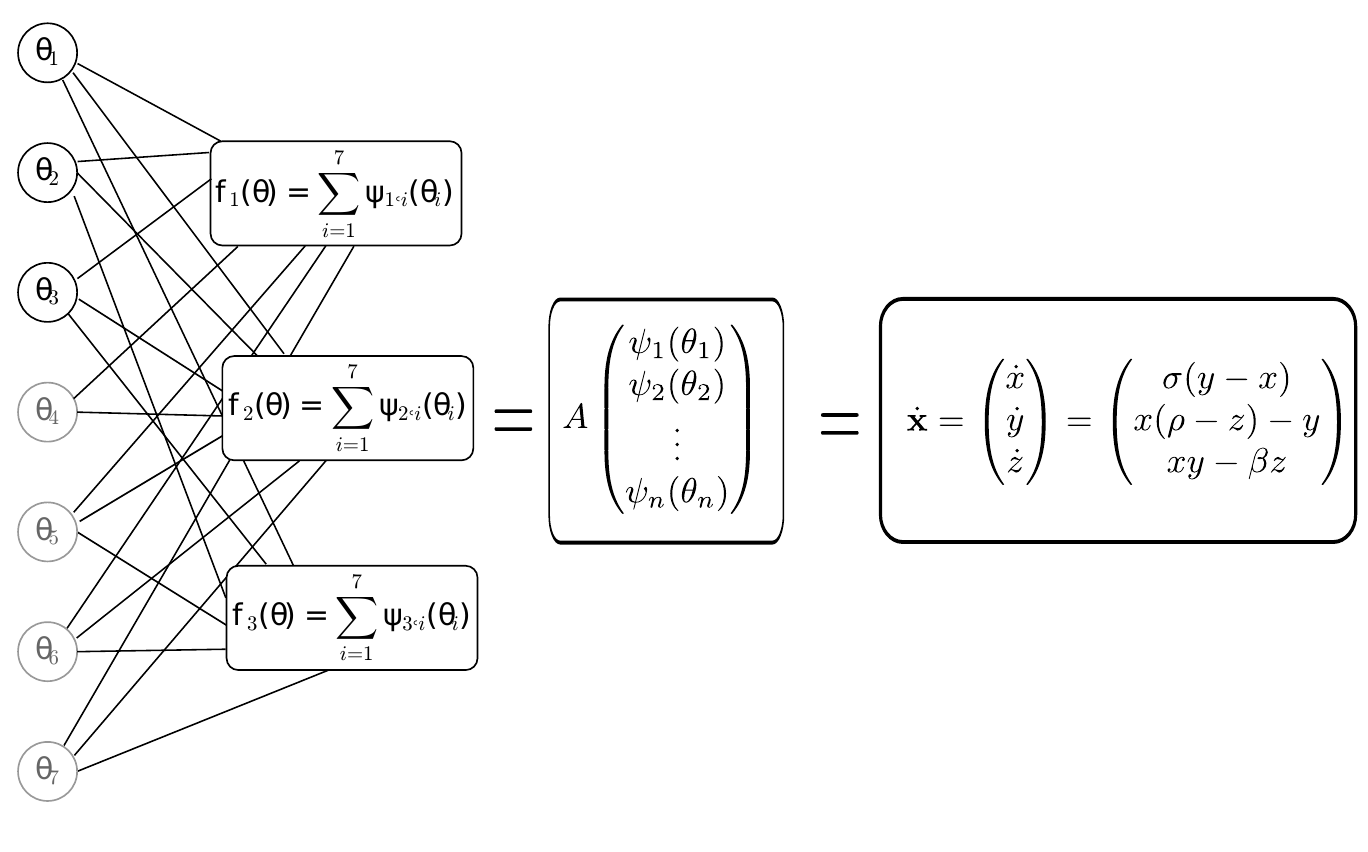}
    \caption{Zero-depth KAN architecture for the Lorenz system. Each input node corresponds to a state variable or lifted nonlinear library term ($x$, $y$, $z$, $xy$, $xz$), and each edge encodes a learned univariate spline activation $\psi_{i,j}$. The three output nodes represent the estimated time derivatives ($\dot{x}$, $\dot{y}$, $\dot{z}$). Because the network has no hidden layers, each output is a sum of independently learned univariate functions, which is equivalent to a generalized regression, making the recovered vector field directly interpretable as a symbolic governing equation. Negligible activations are pruned to zero, yielding the sparse symbolic expressions shown on the right.
    }
    \label{fig:zero-depth}
\end{figure}

We employ a Zero-depth KAN architecture as demonstrated applied to the Lorenz system (Figure~\ref{fig:zero-depth}) which is equivalent to a sparse-regression with an MLP, but implemented as a KAN. The key difference between our method and sparse optimization is that KANDy offers greater modelling capabilities. Reducing the degrees of freedom of the grid size and the number of spline knots, we observed that the slopes of the adaptive splines, which in this limited scenario were linear functions, approximately matched the $\beta$ weights learned by Lasso. In this way, our method generalizes sparse regression by leveraging the adaptability of KANs to apply nonlinear transformations to lifted dictionary features that are not specified a priori. Furthermore, it provides tighter bounds on the learning objective by allowing the KAN to learn the underlying vector field, which is not accessible in least-squares fitting

A typical learning mechanism follows the formula from control theory:
$$
\hat{f_\theta}(\mathbf{x})=\text{KAN}_\theta(\mathbf{x})\approx \frac{d\mathbf{x}}{dt}
$$
where the model consumes states as inputs and predicts the derivatives. The approach is similar to \citet{Bagrow2025,Koenig2024,Koenig2025}. However, our key differentiator to lift $\mathbf{x} \mapsto \Theta$ by some lifting function $\Phi$ includes the states and the lifted features.

\begin{equation}
\hat{f_\theta}(\Theta)=\text{KAN}_\theta(\Theta)\approx \frac{d\mathbf{x}}{dt}
\label{eq:kan_learning_objective}
\end{equation}

\begin{figure*}[!htp]
\centering
\begin{tikzcd}
X \arrow[rr,"f"] \arrow[d,"\Phi"'] & & TX \\
\tilde{X} \arrow[urr,"\mathrm{KAN}_\theta"']
\end{tikzcd}
\caption{The KANDy architecture is equivalent to a commutative diagram where the composition of the estimated KAN composes with the lifting map $\Phi$.}
\label{fig:commutative_diagram}
\end{figure*}

Equation~\eqref{eq:kan_learning_objective} shows the learning objective of the KANDy model. The function $\Phi$ maps the observables to a larger linearized space where the KAN can map to the vector field. The learning process is summarized in the commutative diagram seen in Figure~\ref{fig:commutative_diagram}. Commutative diagrams of this type are referred to as ``lifts" where there exists a factorization through a lifting map.
We find that in practice, $TX$ need not be a vector field, but our lifting method applies to more general topological spaces.

\subsection*{Lifted Differential Loss Function}

Training dynamics models by differentiating through an ODE solver is an established practice~\cite{chen2018neuralode}. However, since our model takes in the lifted features, each of those lifted features are calculated during training from integrator steps by factoring through the commutative diagram seen in ~\ref{fig:commutative_diagram}. Given the trajectories $\{\mathbf{x}^{(n)}(t)\}_{n=1}^N$, we define the model-predicted trajectory $\hat{\mathbf{x}}^{(n)}(t)$ as the solution of
$$
\frac{d\hat{\mathbf{x}}^{(n)}(t)}{dt}
=
f_\theta\!\big(\Phi(\hat{\mathbf{x}}^{(n)}(t))\big),
\qquad
\hat{\mathbf{x}}^{(n)}(t_0)=\mathbf{x}^{(n)}_0.
$$

First, we add trajectory-matching objective can be written as
$$
\mathcal{L}_{roll}(\theta)
=
\frac{1}{N}
\sum_{n=1}^N
\frac{1}{t_T - t_0}
\int_{t_0}^{t_T}
\left\|
\hat{\mathbf{x}}^{(n)}(t) - \mathbf{x}^{(n)}(t)
\right\|_2^2 \, dt.
$$

Because $\Theta = \Phi(\mathbf{x})$ includes nonlinear lifted terms, 
the lifting map $\Phi(\cdot)$ is evaluated on the current predicted state at every integration stage. This ensures that the lifted coordinates remain consistent with the evolving trajectory during training. Our training method  optimizes against both derivative supervision for local consistency of the vector field and state rollout supervision for global trajectory consistency. Derivative supervision targets $\dot{\mathbf{x}}_i$ with inputs $\Theta_i=\Phi(\mathbf{x}_i)$, then
$$
\mathcal{L}_{\mathrm{deriv}}(\theta)
=
\frac{1}{N}
\sum_{i=1}^{N}
\left\|
f_\theta(\Theta_i) - \dot{\mathbf{x}}_i
\right\|_2^2.
$$

The full training objective is

\begin{equation}
\mathcal{L}(\theta)=
\mathcal{L}_{\mathrm{deriv}}(\theta) +\lambda_{\mathrm{roll}}\,\mathcal{L}_{\mathrm{roll}}(\theta).
\label{eq:total_loss}
\end{equation}

Equation~\eqref{eq:total_loss} shows the total combination of both the lifted derivative loss and the model rollout loss. The model rollout loss is modified by an additional parameter $\lambda_{\mathrm{roll}}$ which is a training hyperparameters. Other training hyperparameters include, the type of integration to perform on the lifted features (e.g, Runge-Kutta, Euler, Rusanov which were implemented in our experiments, but expands to others).

\subsection*{Chaotic Reconstruction Measures}

The Lyapunov time $\tau_L$  characterizes the fundamental predictability horizon of a chaotic dynamical system. It is defined as the inverse of the largest Lyapunov exponent and represents the characteristic timescale over which trajectories that are initially nearby diverge by an order of magnitude. Beyond the Lyapunov time, small uncertainties in the initial conditions grow exponentially, rendering long-term predictions unreliable regardless of model accuracy. As a result, the Lyapunov time provides a theoretical upper bound on the performance of deterministic forecasting in chaotic systems and serves as a natural benchmark for evaluating predictive performance. The Lyapunov time is given by $\tau_L \;=\; \nicefrac{1}{\lambda_{\max}}$.

The quantity $\lambda_{\max}$ denotes the largest Lyapunov exponent of the system. It measures the average exponential rate at which infinitesimally close trajectories diverge in phase space, such that a small perturbation $\delta \mathbf{x}(0)$ grows as
$$
\lVert \delta \mathbf{x}(t) \rVert \sim \lVert \delta \mathbf{x}(0) \rVert e^{\lambda_{\max} t}.
$$
A positive value of $\lambda_{\max}$ is a defining signature of chaos, indicating sensitive dependence on initial conditions. Consequently, the Lyapunov time $\tau_L = 1/\lambda_{\max}$ sets the characteristic timescale over which predictions remain meaningful.

To quantify the temporal degradation of forecasts, we computed the cumulative normalized RMSE as a function of prediction horizon $t = h \Delta t$:
\begin{equation}
\mathrm{NRMSE}(t) = 
\frac{1}{\sqrt{N}}
\left( 
\frac{1}{t} \int_0^{t} 
\big\| \hat{\mathbf{x}}(\tau) - \mathbf{x}(\tau) \big\|_2^2 
\, d\tau
\right)^{1/2} \bigg/ 
\sigma_{\mathrm{data}},
\label{eq:nrmse}
\end{equation}
where $\hat{\mathbf{x}}(t)$ denotes the prediction of the model and $\sigma_{\mathrm{data}}$ is the root--mean--square of the true signal amplitudes, and $N$ is the dimension of the trajectories.

\subsection*{Function Fitting To Global Activations}\label{sec:complexity}

One of the core impediments we found in estimating governing equations with the standard KAN code implementation released by \citet{liu2025kan} is that fitting functions on edge splines only assesses fit quality on individual edges, making the fitted symbolic formula a local property of each edge. To overcome this, we implemented a custom symbolic formula extraction procedure that operates on the model activations. Most importantly, we found that setting a default fallback to zero on an edge when a function could not be fit improved symbolic extraction for dynamical systems. This zero fallback reduces the inclusion of poorly fitting or spurious functions, thereby decreasing noise in the extracted equations and enhancing both interpretability and sparsity. As a result, the symbolic formulas become cleaner and more aligned with the underlying system dynamics.

Our symbolic formula extraction from activation proceeds as follows; The KANDy model only has one wide layer with one input per lifted term, let $i$ be the index of the lifted term for $i=1,2,\cdots,n$ where $n$ is the total number of model inputs, and $j$ be the index of the output formula (e.g. in Figure~\ref{fig:zero-depth} the input feature $x$ corresponds to the index $i=1$ while the output target $\dot{x}$ corresponds to $j=1$). The KANDy model has no depth; however, we retain the index $l$ as the depth index for generalizability, even though $l=0$ in our implementation. For each edge $(l,i,j)$, we select a symbolic function from a predefined library. We then perform a greedy-edge-wise symbolic fitting.

For each KAN edge $(l,i,j)$, we select a symbolic function from a predefined library $\mathcal{L} = {f_1, f_2, \dots, f_K}$. We then perform a greedy edge-wise symbolic fitting. For each candidate function $f_k \in \mathcal{L}$, we fit $f_k$ to the learned spline activation $\phi_{l,i,j}$ and compute the coefficient of determination $R^2_{l,i,j,k}$. along with an associated complexity measure $c_k$ for each candidate. In our implementation, $c_k$ is an integer where the zero function corresponds to a complexity weight of $0$, and higher polynomial features have a complexity term related to the power of the polynomial, and all transcendental functions have a weight of $3$. The best candidate per edge is selected by maximizing a penalized score that balances fit quality against symbolic complexity:

\begin{equation}
S_{l,i,j} = R^2_{l,i,j} - w \cdot \frac{w_s}{1 - w_s} \cdot c_{l,i,j}
\end{equation}

where $w_s \in (0,1)$ is a weight controlling the preference for simpler expressions and $w > 0$ is an overall regularization strength.

We then apply a two-stage filtering and selection procedure. First, we discard all edges whose best-fit $R^2$ falls below a threshold $\tau$:

\begin{equation}
\mathcal{E}_{\text{eligible}} 
= \left\{ (l,i,j) \;\middle|\; R^2_{l,i,j} \geq \tau \right\}
\end{equation}

We next rank the eligible edges by their score $S_{l,i,j}$ in descending order and retain the top-$T$ edges, optionally subject to a total complexity budget $C_{\max}$:

\begin{equation}
\mathcal{E}_{\text{kept}} \subseteq \mathcal{E}_{\text{eligible}}, \quad |\mathcal{E}_{\text{kept}}| \leq T, \quad \sum{(l,i,j) \in \mathcal{E}_{\text{kept}}} c(l,i,j) \leq C_{\max}
\end{equation}

For each retained edge $(l,i,j) \in \mathcal{E}_{\text{kept}}$, we replace the learned spline activation $\phi(l,i,j)$ with its best-fit symbolic function. All remaining edges not in $\mathcal{E}_{\text{kept}}$ are set to the zero function, effectively pruning them from the network. The resulting model fully captured the symbolic formulas for the dynamical systems we studied. This procedure can be viewed as a complexity-regularized symbolic selection procedure over a discrete symbolic library, while enforcing an explicit goodness-of-fit constraint. Symbols with higher intrinsic complexity must achieve proportionally better predictive accuracy to be selected, and edges that fail to meet the threshold are pruned entirely. 

\subsection*{KANDy}

The Kolmogorov-Arnold Representation Theorem says that for any continuous function $f(x_1,\, x_2, \cdots,\,x_n)$ there is a decomposition that makes this multivariate function the sum of continuous functions applied to sums of univariate functions such that

$$
f(x_1,\, x_2,, \cdots,\,x_n) = \sum^{2n+1}_{q=1}\Phi_q(\sum^{n}_{i=1}\phi_{q,i}(x_n)\big).
$$

\citet{liu2025kan} generalizes the original theorem to arbitrary widths and depths and contextualizes it in today’s deep learning world. However, there is no known purely mathematical generalization of this theorem that expands to compositions of continuous functions that is presently known \cite{liu2024kan2}. E.g., for another continuous function $g$,  whether the composition $f \circ g$ has a KAT decomposition; however, the following relation
$$
f \circ g (x_1,\, x_2,, \cdots,\,x_n) = f( g(x_1,\, x_2,, \cdots,\,x_n)).
$$
For example, $f(x,y)=xy=e^{\log(x) + \log(y)}$ is one realization of the KAT. Choosing $g(x,y)=0$ , another continuous function, breaks this particular realization of KAT breaks. In other words, there is no known theorem guaranteeing that a KAT decomposition of $f$ can be systematically extended or adapted to cover $f\circ g$.

This complicated the KAN formulation because each composition represents a layer of the KAN neural network. The original KAN code provides methods to avoid singularities, but the presence of singularities undermines the idea of a true generalization of this theorem and poses complications for dynamical systems.
d
Instead, we choose to formulate this problem by combining aspects of Koopman operator theory and treating the KAN as a regression. This yields the formula below
$$
f(x_1,\, x_2,, \cdots,\,x_n) = \sum^{m}_{q=1}\psi_q(x_n)
$$
where $m$ is the dimension of the embedding of the function into a Koopman-invariant subspace where the dynamics become approximately linear. This formulation improves on SiNDy because the choice of dictionary with sparsity conditions becomes relaxed and the model learns the $\psi_q$ for $q=1,2,\dots,m$ that adjust the dictionary to become approximately linear.

Further complications for the KAT $\mapsto$ KAN analogy are that network layers and neural network architectures allow for outputs of arbitrary size. For example, consider the Lorenz system:

$$  
\dot{\mathbf{x}} =  
\begin{pmatrix}  
\dot{x} \\  
\dot{y} \\  
\dot{z}  
\end{pmatrix}  
=  \begin{pmatrix}  
\sigma (y - x) \\  
x(\rho - z) - y \\  
xy - \beta z  
\end{pmatrix}.  
$$  
  
The state vector is:  
$$  
\mathbf{x}(t) =  
\begin{pmatrix}  
x(t) \\  
y(t) \\  
z(t)  
\end{pmatrix}.  
$$  
  
This can be written compactly as:  
  
$$  
\dot{\mathbf{x}} = f(\mathbf{x}), \qquad  
f : \mathbb{R}^3 \to \mathbb{R}^3.  
$$

If we write  
  
$$  
f(\mathbf{x}) =  
\begin{pmatrix}  
f_1(\mathbf{x}) \\  
f_2(\mathbf{x}) \\  
f_3(\mathbf{x})  
\end{pmatrix},  
$$  
then each component function is  
  
$$  
\begin{aligned}  
f_1(x,y,z) &= \sigma (y - x), \\  
f_2(x,y,z) &= x(\rho - z) - y, \\  
f_3(x,y,z) &= xy - \beta z.  
\end{aligned}  
$$
The codomain of this function is $\mathbb{R}^3$, and in the traditional interpretation of KAT, the codomain is $\mathbb{R}$. To see why this presents problems for the discovery of governing equations of dynamical systems, observe that the decompositions of $f_1$,$f_2$,$f_3$ are learned independently with no guarantee that the shared variables. This becomes particularly evident when examining a component with a nonlinear term; consider $f_2$ or $f_3$ with a nonlinear component and, without loss of generality, consider $f_3$.

Then $f_3 = xy - \beta z$ and applying KAT to this function, there exists a decomposition of continuous functions $h_{i,j},k_{i,j}$ for $i=1,2,\cdots 7$,  such that  

$$f_3 = xy - \beta z = \sum_{i=1}^7h_i(k_{1,i}(x) + k_{2,i}(y) + k_{3,i}(z)).
$$

Since this decomposition holds for all $x,y, z$, let $z=0$. 

Now consider that the inputs to the KAN are the states $x$, $y$, and $z$. For this approximation to occur, there must be at least one neuron that multiplies $x$ and $y$, and we will show the bilinear term $xy$ does not admit a 1D spline approximation. For this to occur there must be continuous functions of $h,u,v:[0,1]\mapsto \mathbb{R}$ such that $f(x,y) = xy = h(u(x) + v(y))$ where this representation is less than the terms provided by the KAT theorem and this neuron must be present in the KAN representation to successfully model this term in the $f_3$ output of the KAN. Theorem~\ref{thm:single_neuron} shows that a single KAN neuron is incapable of modeling this term.

For simplicity, we assume that the neuron splines are modelled as polynomials $c,\, x,\, x^2,\, \dots,\, x^n$ for some $n\in\mathbb{Z}$. The justification of this choice is that a) no non-polynomial terms appear in $f_1$, $f_2$, or $f_3$, and b) the edge-wise trade-off of function fit quality $R^2$ and complexity score $C_k$, for $k=1,2,3$ in this case, and d) the introduction of transcendental function and the inverse function of said transcendental functions requied to recover terms blows up in terms of complexty and introduces singularities and possibly an explosion in the number of neurons requried.

Introducing these constraints enables the accurate determination of the limitations of KANs in approximating the equations of the Lorenz system. For example, introducing arbitrary depths with no hidden width yields a similar impossibility for the Lorenz system (Theorem~\ref{thm:deep_kan_xy}), thus requiring additional hidden width, e.g., the network must be at least $[3, 3, \dots, 3]$. \citet{liu2025kan} shows that a KAN models $xy$ by allowing for hidden square neurons and pass through neurons $\pm $ $h^2$ because $(x+y)^2 - (x^2 + y^2) = 2xy$. The derived network in this case provides clues as to why Koopman lifting is required for the Lorenz system. Recovering the nonlinear terms requires the introduction of additional square terms, a result we formalize in Theorem~\ref{prop:quadratic_obstruction}.

The obstruction for width-$3$ KANs is subtler than in the width-$1$ case. Indeed, a single bilinear term can be realized by polarization $xy=\frac14\big((x+y)^2-(x-y)^2\big)$. Thus, a width-$3$ KAN with a monomial dictionary can represent $xy$ exactly. However, the full Lorenz system requires two independent bilinear terms, namely, $xy$ and $xz$, and this simultaneous requirement leads to a genuine rank obstruction. Thus, recovering the governing equations from a deep KAN with assumed affine polynomial edge functions introduces an algebraic obstruction, which we formalize in Theorem~\ref{thm:inductive_width3}.

\section*{Results}

In this section, we evaluate KANDy across discrete and continuous dynamical systems, including the H\'enon and Lorenz systems, as well as chaotic PDE benchmarks, e.g., the KS equation and the Inviscid Burgers equation. We begin with our toy example of the Hopf fibration as a canonical nonlinear quotient map, demonstrating the model’s ability to learn nontrivial fiber-bundle structure and recover symbolic invariants.

\stepcounter{section}
\setcounter{equation}{0}

\subsection*{Hopf-Fibration}

Symmetry plays a central role in the mathematical structure of physical laws and dynamical systems, and group actions encoding symmetry arise naturally in systems ranging from classical mechanics and field theory to fluid dynamics and pattern-forming partial differential equations. These symmetries introduce redundancy in the state description: multiple configurations correspond to the same physical state. As a result, meaningful observables and reduced models are most naturally defined on quotient spaces obtained by identifying symmetry-related states~\cite{kiani2024hardness, perin2025ability}.

Classical approaches to symmetry reduction rely on analytical insight into the governing equations, leading to invariant variables, conserved quantities, or reduced coordinates. In data-driven settings, however, such structures may be unknown or difficult to derive explicitly. While recent advances in equivariant and invariant neural networks enable the incorporation of known symmetries into model architectures, they do not address the inverse problem of learning the invariants themselves from data.

\begin{figure*}[t]
\vskip 0.15in
\includegraphics[width=0.9\textwidth]{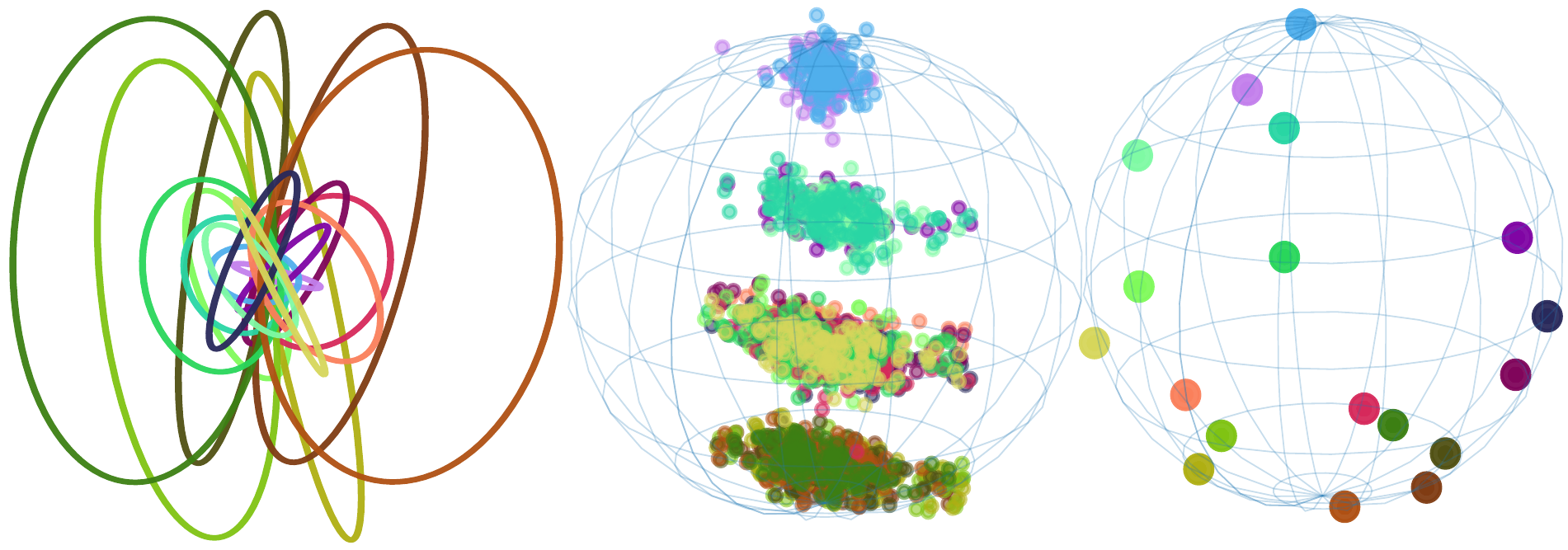}
\caption{
Left: Fibers of the Hopf fibration visualized by stereographic projection of $S^3$ into $\mathbb{R}^3$. Each closed curve corresponds to a $U(1)$ orbit under the group action $(z_1,z_2)\mapsto(e^{i\theta}z_1,e^{i\theta}z_2)$ and each closed curve is colored. Center: a deep KAN of two hidden layers, each with four terms. The learned embedding model scatters the fibres on the sphere.  Right: Outputs of a trained KANDy model evaluated along the same fibers. Fibres of the closed curve are colored (left) correspond to their respective fibres (center and right).
}
\label{fig:hopf_main}
\vskip -0.1in
\end{figure*}

We begin by studying a canonical geometric example: the Hopf fibration. The Hopf map realizes the quotient of the three-sphere $S^3$ by a $U(1)$ action, producing the two-sphere $S^2$. Learning this map requires capturing a nontrivial fiber bundle structure and enforcing invariance under group orbits. We show that KANDy can learn the Hopf quotient map from data, preserve fiber constancy, and admit explicit symbolic representations that align with the underlying invariant theory, whereas a typical KAN may struggle to do so. Figure~\ref{fig:hopf_main} shows both KANDy and a KAN applied to the Hopf-fibration. The standard KAN model collapsed the fibres of the orbits in the Hopf fibration along the azimuth direction. The learned representation collapses entire fibers to single points on $S^2$, demonstrating invariance along group orbits. Both models were fit using a modified radial basis function $e^{x^2}$ as the default spline with a grid of 64 with 7 knots per spline.

The KAN had inputs $x_1, x_2, \dots x_4$ with three layers each 4 summand terms wide, and after training yielded a mean angular error $1.02$, the upper 95-th percentile radial error of $2.09$, mean fiber RMS=$0.124$ and a maximum fibre error of $0.147e$. On the other hand, KANDy had lifted inputs $(x_1x_3, x_2x_4, x_2x_3, x_1x_4, x_1^2 + x_2^2 - x_3^2 - x_4^2) \in \mathbb{R}^5$ making a higher dimensional lift using graded component, which after training gave a mean angular error of $1.80e-05$,  the upper 95-th percentile radial error of $0.0$, mean fiber RMS=$3.13e-07$ and a maximum fibre error of $4.77e-07)$.

\begin{equation}
\phi_1 \approx 2.0715\times10^{-4}, \qquad
\phi_2 \approx 3.9872\times10^{-3}, \qquad
\phi_3 \approx 8.6808\times10^{-3}.
\label{eq:kan_raw_4d}
\end{equation}

Equation~\eqref{eq:kan_raw_4d} shows the learned formula with the typical formula loss function and composition. For a KAN of width $[4,4,4,3]$, the symbolic extraction identifies several edge-wise nonlinearities in the early layers. In particular, multiple input-to-hidden edges in the first hidden layer are well approximated by cosine functions (e.g., $R^2 = 0.9986, 0.9978, 0.9964$), and additional hidden-to-hidden edges are fit by sine and cosine functions with coefficients of determination up to $R^2 = 0.9990$ and complexity $c = 2$. An intermediate hidden-to-hidden edge is approximated by a linear function ($x$) with $R^2 = 0.9657$ and $c = 1$. Symbolic extraction was performed with a simplification weight of $0.8$ and an $R^2$ threshold of zero. All of which created a zeroed-out final function composition with only constant terms. These experiments were performed with the default symbolic discovery routine with default parameters from the PyKAN library~\cite{liu2025kan}. The discovered formula~\eqref{eq:kan_raw_4d}, which is linear, confirms the model predictions of the standard deep-KAN shown as the center plot in Figure~\ref{fig:hopf_main}, where the fibers appear evenly scattered but collapsed in the azimuth direction corresponding to the order of magnitude reduction in $\phi_1$.
\begin{equation}
\begin{cases}
\phi_1 \approx 0.4077\,x_1x_3 + 0.4078\,x_2x_4 + 0.00139 \\
\phi_2 \approx 0.4082\,x_2x_3 - 0.4100\,x_1x_4 + 0.00854 \\
\phi_3 \approx 0.5768\,(x_1^2 + x_2^2 - x_3^2 - x_4^2) - 0.00278 
\end{cases}
\label{eq:kan_phi_5d}
\end{equation}

Eq.~\eqref{eq:kan_phi_5d} shows the simplified symbolic expressions extracted from KANDy. The model was trained with the typical MSE loss function, as there is no derivative or dynamics, and the composition of the same Hopf-fibration under the KANDy assumptions of additional non-linear lifted homogeneous terms. The symbolic extraction identifies predominantly polynomial edge-wise non-linearities across the first layer. In particular, multiple input-to-hidden edges are accurately approximated by linear ($x$) and quadratic ($x^2$) functions introducing additional nonlinearities for already nonlinear inputs, with $R^2 \approx 0.99999$–$1.00000$. Linear edges correspond to default complexity $c=1$, while $x^2$ are selected with default complexity $c=2$, indicating that second-order polynomial structure is sufficient to represent the learned transformations. These edge-wise polynomial mappings compose through the network to produce the final symbolic expressions, which are dominated by linear combinations of the input variables with small quadratic correction terms. Symbolic extraction was performed with default settings, using a simplification weight of $0.8$ and an $R^2$ threshold of $0$, yielding a compact yet high-fidelity representation of the learned mapping. The presence of quadratic and linear terms in the hidden layers indicates that the model is learning a quadratic form on the manifold of activations, which explains why KANs with multiplication layers may fail. 

Figure~\ref{fig:hopf_main}, Equations~\eqref{eq:kan_raw_4d} and \eqref{eq:kan_phi_5d}, and subsequent analysis illustrates that KANDy learns a nontrivial topological properties from data juxtaposed with standard KANs, showing that depth alone is insufficient for some non-linear systems. Motivated by this geometric result, we turn to dynamical systems and partial differential equations with continuous symmetries. The Kuramoto–Sivashinsky equation exhibits translation symmetry and chaotic dynamics on a quotient space, while Burgers’ equation possesses Galilean invariance. The Lorenz system, though finite-dimensional, exhibits a discrete symmetry that yields invariant polynomial coordinates. In each case, we frame the learning task as the discovery of a map from the original state space lifted to a larger, possibly infinite dimensional, space, symmetry-invariant representation.

\subsection*{Discrete Dynamical Systems}

\begin{figure}[!htpb]
  \centering
  \includegraphics{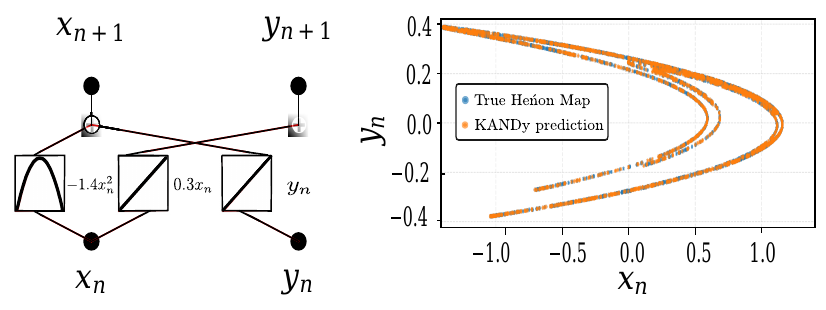}
  \caption{Left: KANDy architecture for the Henon map, which recovers the governing equations. Right: the true attractor overlaid with the KANDy learned attractor.}
  \label{fig:henon}
\end{figure}

We next consider the Hénon map as a canonical discrete-time chaotic system with known governing equations. Fig.~\ref{fig:henon}, the learned KANDy accurately reproduces the characteristic folded structure of the Hénon attractor, demonstrating that even shallow architectures suffice for discrete-time systems. The KANDy network plot reveals that the learned univariate activation functions align with the functional form of the true Hénon map. The model was fit with a radial basis function, had a grid size of $5$ with $3$ spline knots. The model reduced several orders of magnitude in both training and testing loss. 

\begin{equation}
\begin{cases}
x_{n+1} = -1.4\,x_n^{2} + 1.0\,y_n + 1.0 \\
y_{n+1} = 0.3\,x_n
\end{cases}
\label{eq:henon}
\end{equation}
The learned governing equations (Eq.\eqref{eq:henon}) match the true parameters of the Henon system with 100\% fidelity. A similar result appears in \citet{Panahi2025KANModelDiscovery}; however, the estimated formula is not optimized for statistical attractor properties and is reported as a standard KAN, and this formulation lacks lifted features. However, the KANDy method successfully identified the governing equations of the Ikeda Optical Cavity model, which is a known edge case for sparse regression techniques and is not successfully captured in \citet{Panahi2025KANModelDiscovery}. For brevity, this experiment is omitted; the full experiment is available in the supplementary material and on GitHub.\footnote{\href{https://github.com/Center-For-Complex-Systems-Science/kandy/blob/main/examples/ikeda_example.py}{Full experiment}}

\subsection*{Continuous Dynamical Systems}

Continuous-time dynamical systems provide a natural testbed for evaluating whether a learned model can recover underlying governing equations while preserving long-horizon qualitative behavior. Unlike discrete mappings, continuous systems require the model to remain consistent under numerical integration and to respect stability, dissipation, and invariant structures such as attractors. In this section, we focus on benchmark systems defined by ordinary differential equations, using them to assess both the accuracy of learned vector fields and the fidelity of the resulting trajectories.


The Lorenz system is a canonical example of a low-dimensional chaotic dynamical system that serves as a benchmark for model discovery. Accurately recovering the Lorenz dynamics requires capturing both the local vector field and the global structure of the strange attractor. The KANDy architecture successfully captures both. As such, performance on the Lorenz system provides insight into the model’s ability to generalize beyond training trajectories, maintain stability under rollout, and recover interpretable governing equations consistent with known physics.

\begin{figure*}[!htbp]
  \centering
  \includegraphics[width=1.0\linewidth]{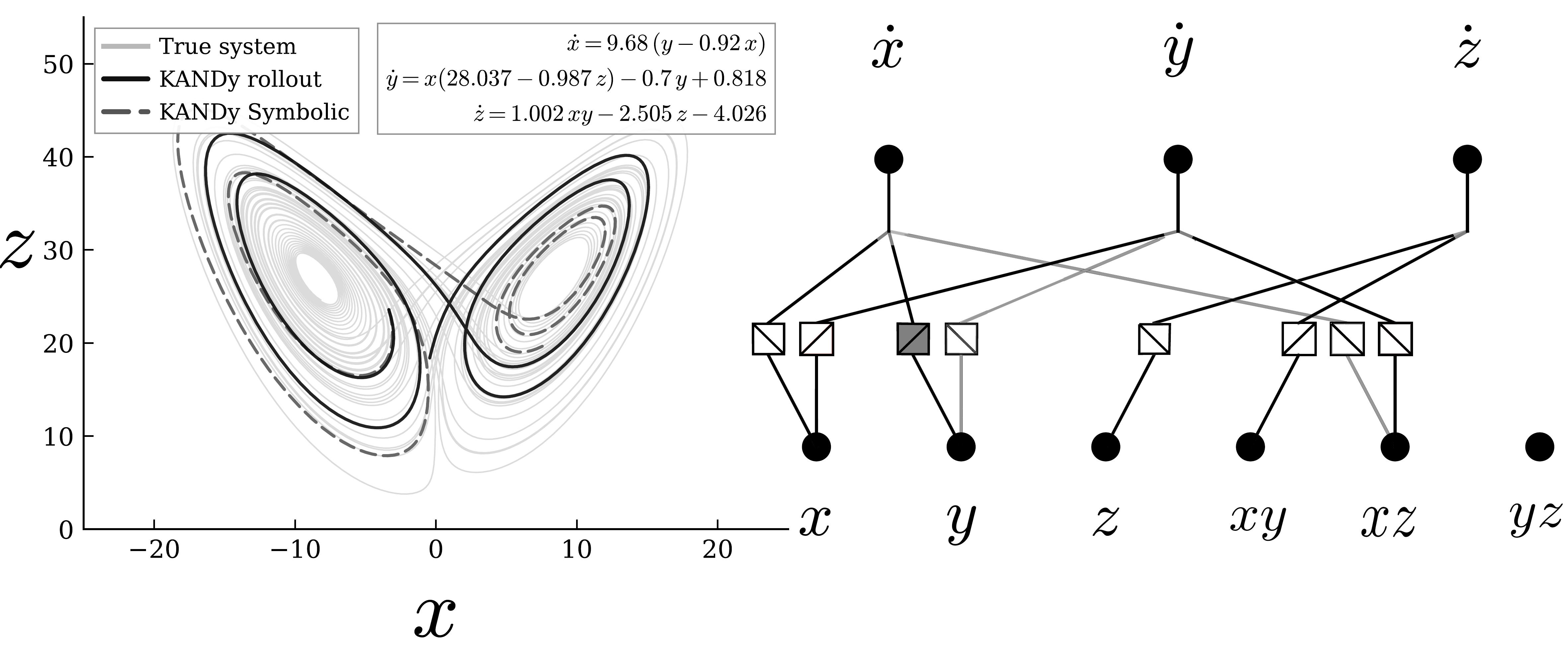}
  \caption{Left: a trajectory learned by KANDy (solid grey) and the trajectory estimated from the learned governing equations (dotted grey) with the learned governing equation in the legend overlayed on the true attractor. Right: the KANDy architecture, with idealized representations of the learned activations of the model applied to non-linear lifted terms.
  }
  \label{fig:lorenz}
\end{figure*}

Figure~\ref{fig:lorenz} shows a learned trajectory from the model as well as the learned governing equation overlaid on the true attractor. The model was trained on the classic Lorenz parameters $\sigma=10.0$, $\rho=28.0$, and $\beta=\nicefrac{8.0}{3.0}$ over time of $50$ steps with $dt=0.005$, translating to 10,000 samples with a burn-in period of 2. The model was trained for 300 epochs with a 4th-order Runge-Kutta integrator, integrating 10 steps ahead, with subsequent grid updates every 50 steps up to 300 steps, using a learning rate of $10^{-3}$. The default spline function instantiating the model is a modified radial basis function $e^{-3x^2}$ which was chosen because of the difficulty fitting the linear damping term $\nicefrac{dy}{dt}=-y$ which often required longer training horizons as well as more ``spread-out" spline initializations. The model parameters included a width of 6 inputs (one for each term) and 3 outputs representing the KANDy architecture with a value of spline knots $k=1$ and a grid-size of 7. The training set initialized with a loss of 1828.79 and finished training with a loss of 0.47. For the out-of-training test set the training started at a loss of  1765.44 and completed at 0.45. The rolling integrator loss began training at 2.53 and completed on 0.02. In terms of order of magnitude, the training and out-of-training test losses were each reduced by about four orders of magnitude, while the rolling integrator loss was reduced by about two orders of magnitude. All losses were calculated as the MSE loss.

\begin{figure*}[!htbp]
  \centering
  \includegraphics[width=\linewidth]{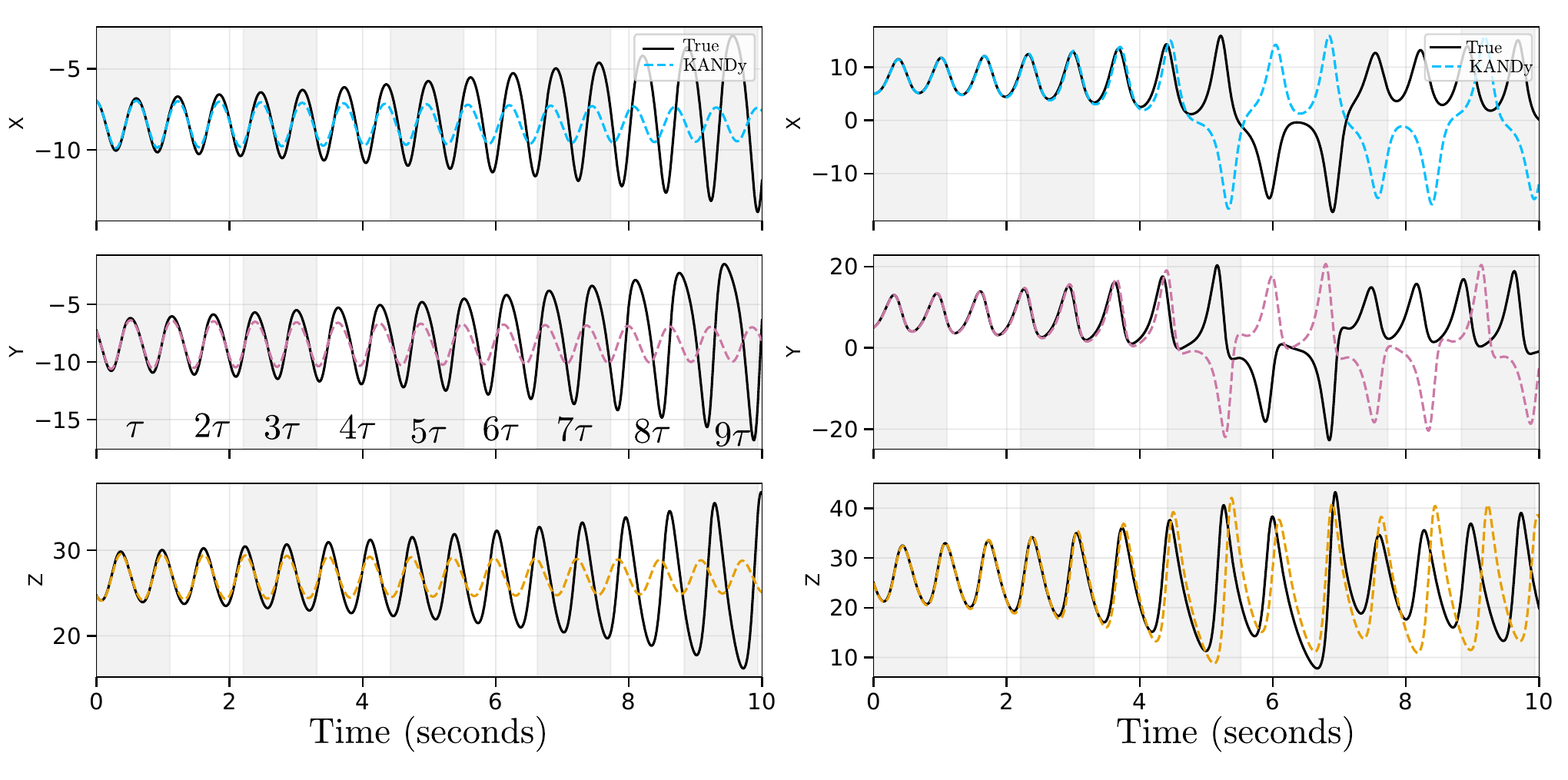}
  \caption{Left: KANDy rollout plotted against Lyapunov for in-training trajectories. Right: KANDy rollout with an out-of-training distribution as the initial condition. The KANDy rollout remains in both phase and amplitude coherence with the attractor for 5 Lyapunov times, $\tau$.}
  \label{fig:lorenz_rollout_in_lyanpunov}
\end{figure*}

Figure~\ref{fig:lorenz_rollout_in_lyanpunov} shows the KANDy rollout plotted against Lyapunov for in-training trajectories with initial condition $(0.0, 0.0, 0.0)$ and  on an out-of-training distribution initial condition $(5.0,\, -25.0,\, 1.0)$. The KANDy rollout stays in synchrony with the attractor for 5 Lyapunov times $\tau$.

To assess generalization, we perform few-shot rollouts from initial conditions sampled outside the training distribution but still within the basin of attraction. As shown in Fig.~\ref{fig:lorenz}, the KAN successfully converges to the correct attractor after a transient phase, despite not having observed these initial conditions during training. This behavior highlights that the learned model captures the attractor manifold itself, rather than a narrow subset of trajectories.

\begin{figure*}[!htp]
    \centering
    \includegraphics[width=1.0\linewidth]{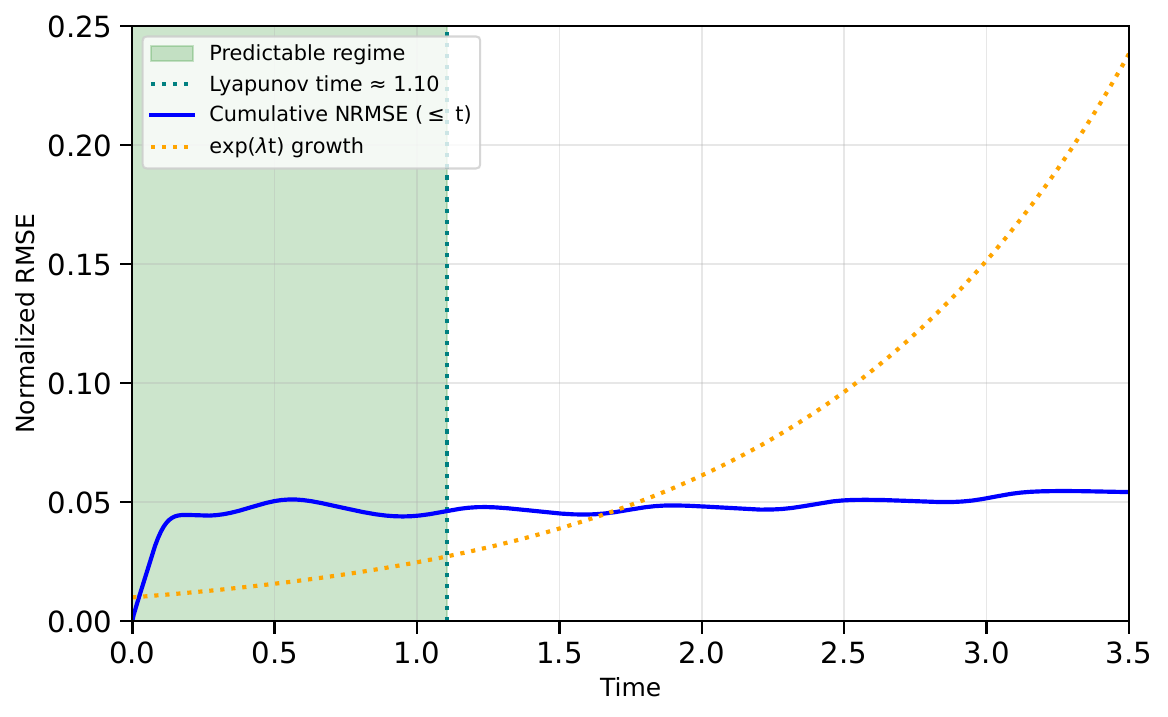}
    \caption{The NRMSE for one out-of-distribution initial condition. The period before one Lyapunov time passing colored in green shows that the KANDy model alone is sufficient to approximate the vector field within the predictability limits.}
    \label{fig:lyapunov_curve}
\end{figure*}

Figure~\ref{fig:lyapunov_curve} shows NRMSE for one out-of-distribution initial condition, e.g. $(25.0,\, 5.0,\, 4.0)$. The period before one Lyapunov time passing colored in green shows that the KANDy model alone is sufficient to approximate the vector field within the predictability limits, and shows the resulting error curve (solid blue) together with the Lyapunov time $\tau_L$ (dotted vertical line).  The normalized error remains below $0.1$ for $t < 0.8$, begins to grow rapidly near $t \approx 1.0$, and crosses $0.4$--$0.5$ by $t \approx 2.0$.  This transition is closely aligned with the theoretical Lyapunov time $\tau_L$, confirming that the predictability horizon of the model is fundamentally limited by the chaotic divergence rate of the underlying dynamics.

The combined qualitative and quantitative results demonstrate that the trained model achieves accurate short--term trajectory forecasting up to roughly one Lyapunov time, beyond which exponential divergence in phase leads to decorrelation while preserving the attractor geometry.  This outcome is consistent with the theoretical limit of deterministic predictability for chaotic systems, validating the model’s fidelity to the true Lorenz dynamics.

Expanding the right-hand sides of the classic Lorenz equations~\eqref{eq:lorenz} makes the parameter-to-coefficient mapping explicit:
\begin{equation}
\dot{x} = -\sigma x + \sigma y,\qquad
\dot{y} = -xz + \rho x - 1\cdot y,\qquad
\dot{z} = 1\cdot xy - \beta z.
\label{eq:expanded}
\end{equation}

The estimated governing equations of the Lorenz system were 

\begin{equation}
\begin{aligned}
\dot{x} &= -8.855\,x + 9.646\,y, \\
\dot{y} &= -0.972\,x z + 27.378\,x - 0.7\,y + 0.178, \\
\dot{z} &= 1.0\,x y - 2.702\,z + 0.859.
\label{eq:estimated}
\end{aligned}
\end{equation}

Comparing Eq~\eqref{eq:estimated} term-by-term with Eq~\eqref{eq:expanded} gives the implied Lorenz parameters (and highlights deviations from the canonical structure): In Eq~\eqref{eq:expanded}, the coefficients of $x$ and $y$ in $\dot{x}$ must be $-\sigma$ and $+\sigma$, i.e.\ equal magnitude and opposite sign. The model \eqref{eq:estimated} instead yields $\sigma_x \approx 8.855$ (from $-8.855\,x)$, $\sigma_y \approx 9.646$ (from  $+9.646\,y$).

A natural single estimate is their average, $\hat{\sigma}\approx (8.855+9.646)/2 = 9.2505$. The internal consistency error for the Lorenz form is the mismatch $\Delta\sigma := 9.646 - 8.855 = 0.791$, which corresponds to a relative discrepancy of about $\nicefrac{\Delta\sigma}{\hat{\sigma}} \approx \nicefrac{0.791}{9.2505} \approx 8.6\%$. Thus, the inferred $\sigma$ is close to the classical value $\sigma=10$ (absolute error $\approx 0.75$ if using $\hat{\sigma}$, i.e.\ $\approx 7.5\%$ relative), but the two linear coefficients in $\dot{x}$ do not satisfy the exact Lorenz constraint of equality.

In \eqref{eq:expanded}, the coefficient multiplying $x$ in $\dot{y}$ is $\rho$. The estimate gives $27.378\,x$, hence $\hat{\rho}\approx 27.378.$ Compared to the common chaotic parameter $\rho=28$, the absolute error is $|\hat{\rho}-28| = 0.622$, which is a relative error of $\nicefrac{0.622}{28}\approx 2.2\%$. This is well within a few percent tolerance used to judge the recovery of Lorenz parameters from noisy data.

In \eqref{eq:expanded}, the linear damping coefficient in $\dot{z}$ is $-\beta$. The estimate gives $-2.702\,z$, so $\hat{\beta}\approx 2.702$. Relative to the classical $\beta=8/3\approx 2.6667$, the absolute error is $|\hat{\beta}-8/3| \approx 0.0353$, and the relative error is $ \nicefrac{0.0353}{2.6667}\approx 1.3\%$. This indicates an excellent match.

The Lorenz model requires coefficients $-1$ for $xz$ in $\dot{y}$ and $+1$ for $xy$ in $\dot{z}$. The estimate provides $-0.972\,xz$ and $1.0\, xy$: $\text{coeff}(xz \text{ in } \dot{y}) = -0.972$ which is a 2.8\%  deviation from $-1$, coeff($xy$ in  $\dot{z}) = 1.0$ implying 0\% deviation from $1$ (to reported precision). Hence, the recovered quadratic coupling is very close to the canonical Lorenz structure.

In \eqref{eq:expanded}, the $y$-coefficient in $\dot{y}$ is exactly $-1$.
The estimate instead gives $-0.7\,y$, which is a substantial deviation: $|-0.7 - (-1)| = 0.30$ which is a 30.0\% relative error in that coefficient.

In addition, the estimated system includes constant forcing terms $+0.818$ in $\dot{y}$ and $-4.026$ in $\dot{z}$, whereas the standard Lorenz--63 system has no constant terms. Such offsets commonly arise from (i) a nonzero mean in the data if the variables were not centered, (ii) model-form mismatch absorbed by lower-order terms during regression, or (iii) bias due to noise/finite data and regularization. If the true underlying dynamics are Lorenz--63, then subtracting empirical means (or including a constraint that eliminates constants) would typically drive these constants toward zero.

If one adopts a tolerance of, say, $\le 5\%$ relative error for identifying Lorenz parameters from data, then the estimates for $\rho$ ($\sim2.2\%$), $\beta$ ($\sim1.3\%$), and the quadratic coefficients (0--2.8\%) are comfortably within tolerance, while $\sigma$ is near but slightly outside depending on how it is inferred (the two $\dot{x}$ coefficients disagree by $\sim 8.6\%$), and the linear damping in $\dot{y}$ is clearly outside tolerance (about $50\%$). Therefore, the learned model largely recovers the key Lorenz couplings and parameters associated with chaotic behavior (notably $\rho$ and $\beta$ and the $xz,xy$ terms), but it does not perfectly reproduce the exact Lorenz linear structure, suggesting either estimation bias, insufficient/uncentered data, or that the fitted system represents a close-but-not-identical dynamical system in the Lorenz family.


\subsection*{Partial Differential Equations}

Next, PDEs are a substantially more challenging class of chaotic systems, as they describe the evolution of spatiotemporal fields governed by both local interactions and global constraints. In contrast to ordinary differential equations, PDEs involve infinite dimensionality with many degrees of freedom. Successfully learning PDE dynamics, therefore, requires models to capture not only temporal evolution but also spatial structure, locality, and invariances arising from the underlying physics. In this section, we consider benchmark PDE systems to evaluate KANDy's ability to recover governing equations.

\subsubsection*{KS-Equation}

The KS equation is a prototypical nonlinear PDE that exhibits spatiotemporal chaos arising from the interaction of instability, dispersion, and dissipation. Despite its relatively simple analytic form, the KS equation exhibits complex dynamics, including chaotic patterns, broadband spectra, and sensitive dependence on initial conditions. These properties make it a demanding benchmark for model discovery and long-horizon prediction. Accurately modeling the KS equation requires capturing higher-order spatial derivatives, nonlinear coupling terms, and the balance between energy injection and dissipation that governs the system’s attractor. Performance on this system thus provides a stringent test of whether the KANDy can recover interpretable governing dynamics while maintaining stability and fidelity in extended spatiotemporal rollouts.

\begin{figure*}[!htpb]
    \centering
    \includegraphics[width=1.0\linewidth]{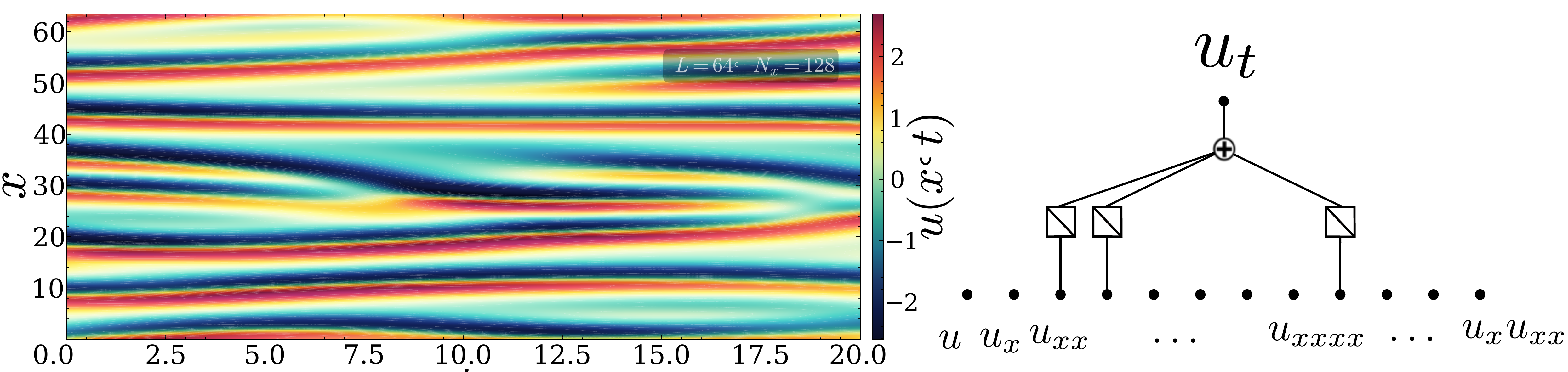}
    \caption{Left: the attractor for the KS space time. The KANDy architecture for the KS equation. The rollout of the trained model is stable and adheres to the attractor. Right: the KANDy architecture showing the lifted features with a robust choice of choices of lifts. Many spurious inputs were ''zeroed" out during training.}
    \label{fig:ks_spacetime}
\end{figure*}

Figure~\ref{fig:ks_spacetime} shows the space-time rollout after training the KANDy (left) model for the KS equation as well as the architecture after symbolic fitting (right), where the 1D splines are replaced with their linear surrogates evaluated on the following clusters of derived inputs: linear terms ($u$, $u_x$, $u_{xx}$, $u_{xxxx}$), purely nonlinear terms ($u^2$, $u^3$, $u_x^2$, $u_{xx}^2$), and mixed terms ($uu_x$, $uu_{xx}$, $uu_{xxxx}$, $u_xu_{xx}$).

When working with Equations~\eqref{eq:ks_governing}, one often considers a parameterized KS form
\begin{equation}
u_t = -\alpha\,u u_x - \nu\,u_{xx} - \kappa\,u_{xxxx} + c,
\label{eq:ks_param}
\end{equation}
where $\alpha$ controls the strength of the nonlinear advection, $\nu$ the (typically destabilizing) second-derivative term, $\kappa$ the (stabilizing) hyperdiffusion, and $c$ is a constant forcing/offset term (which is usually $0$ in the standard nondimensional KS equation).

\begin{equation}
u_t = -0.839\,u\,u_x - 0.459\,u_{xx} - 1.179\,u_{xxxx} - 0.002.
\label{eq:ks_discovered}
\end{equation}
Equation~\eqref{eq:ks_discovered} are the estimated governing equations from the KANDy model. Comparing \eqref{eq:ks_discovered} with \eqref{eq:ks_param} gives the implied parameters $\hat{\alpha} = 0.839$, $\hat{\nu} = 0.459$, $\hat{\kappa} = 1.179$, and $\hat{c} = -0.002$.

If the reference model is the standard nondimensional KS equation \eqref{eq:ks_governing}, the target coefficients are $\alpha=\nu=\kappa=1$ and $c=0$. The relative (percent) deviations of the identified coefficients from unity are therefore $|0.839-1| \approx 16.1\%$ (nonlinear term $uu_x$), $|0.459-1| \approx 54.1\%$  (second derivative  $u_{xx})$, $|1.179-1| \approx 17.9\%$ (fourth derivative $u_{xxxx}$), and the constant offset has absolute magnitude $|c| = 0.002$, which is small in absolute terms (and typically negligible if $u$ and its derivatives are $\mathcal{O}(1)$ after nondimensionalization).

Unlike the Lorenz case, the KS equation is a PDE whose numerical coefficients depend strongly on the chosen nondimensionalization and on the scaling of space and time in preprocessing. In particular, if one rescales variables via $x = L\,\tilde{x}$, $t = T\,\tilde{t}$, $u = U\,\tilde{u}$, then derivatives transform as $u_x \sim U/L$, $u_{xx}\sim U/L^2$, and $u_{xxxx}\sim U/L^4$, meaning the apparent coefficients in front of $u_{xx}$ and especially $u_{xxxx}$ can change substantially with $L$ and $T$ because of the aforementioned invariance properties. For example, if $L=2$ (doubling the length scale), the coefficient multiplying $u_{xx}$ would become one fourth its original value, since $u_{xx}$ scales as $1/L^2$. Thus, even a coefficient mismatch (e.g., $\hat{\nu}=0.459$ instead of $1$) may reflect a different effective scaling of $(x, t)$ rather than an incorrectly discovered governing equation.

From a structural standpoint, \eqref{eq:ks_discovered} recovers exactly the expected KS library terms: the nonlinear convection $u u_x$, the second derivative $u_{xx}$, and the fourth derivative $u_{xxxx}$. This is the key qualitative signature of KS dynamics. Quantitatively, if one applies a tolerance such as $\le 20\%$ relative error (a common heuristic in data-driven PDE identification when derivatives are noisy), then the nonlinear and fourth-derivative coefficients (16--18\%) would be deemed ``close,'' while the $u_{xx}$ coefficient (54\%) would not. Under a stricter $\le 10\%$ tolerance, none of the three would qualify as close to unity. However, because coefficient values are scale-dependent for PDEs, it is often more appropriate to assess closeness after accounting for the particular nondimensionalization used in generating/normalizing the data (e.g. by comparing against the known coefficients in that same scaled coordinate system).

The presence of the small constant term $-0.002$ is not part of the canonical KS equation~\eqref{eq:ks_governing}. As in many regression-based discovery settings, such a term can appear due to a nonzero mean of $u_t$ or imperfect centering of the data. If the underlying physics is KS with zero forcing, then subtracting the temporal/spatial mean of the data (or enforcing $c=0$ as a constraint) would typically drive this term toward zero.

Overall, the estimated equation \eqref{eq:ks_discovered} captures the correct KS mechanism and structure (nonlinear advection balanced by second- and fourth-order dissipation/instability), with coefficients that are of the right sign and order of magnitude; the quantitative closeness to the ``unit-coefficient'' nondimensional KS form depends on the tolerance used and, crucially, on the scaling conventions applied to the data and derivatives.

\begin{figure*}[!hpt]
    \centering
    \includegraphics[width=0.49\linewidth]{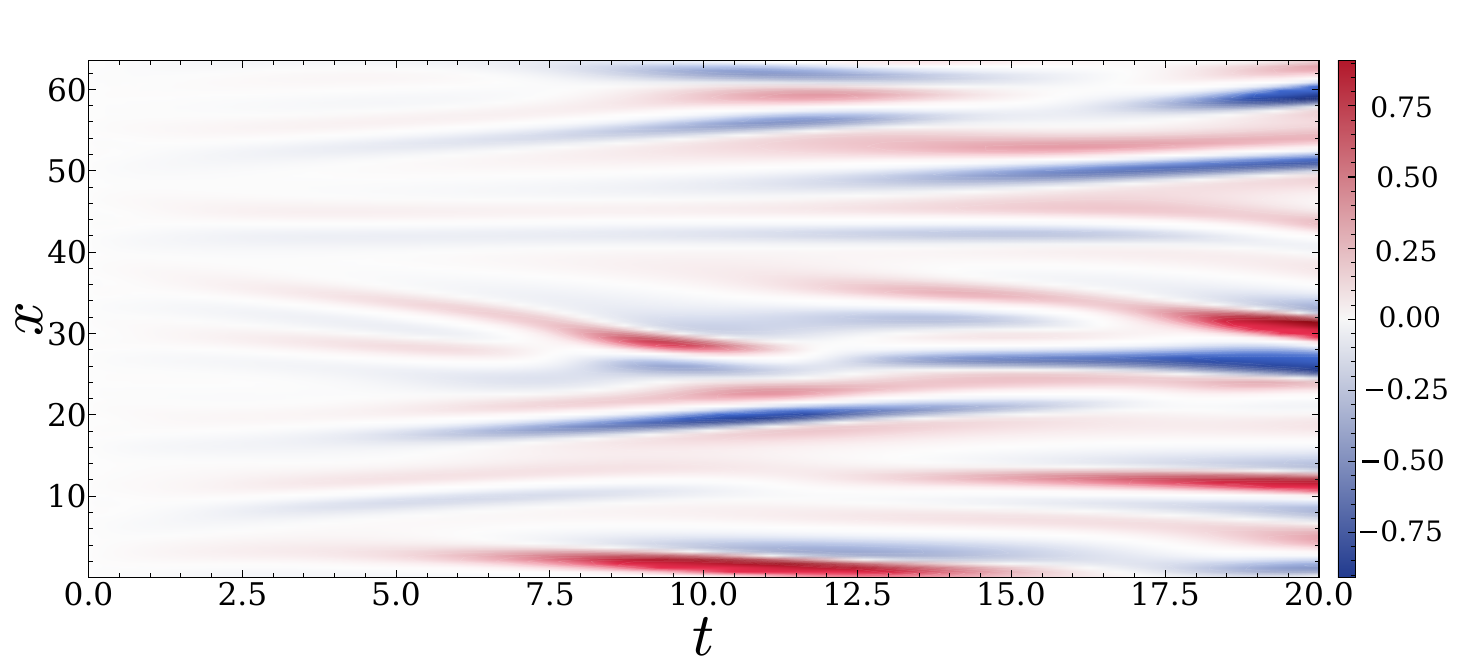}\includegraphics[width=0.49\linewidth]{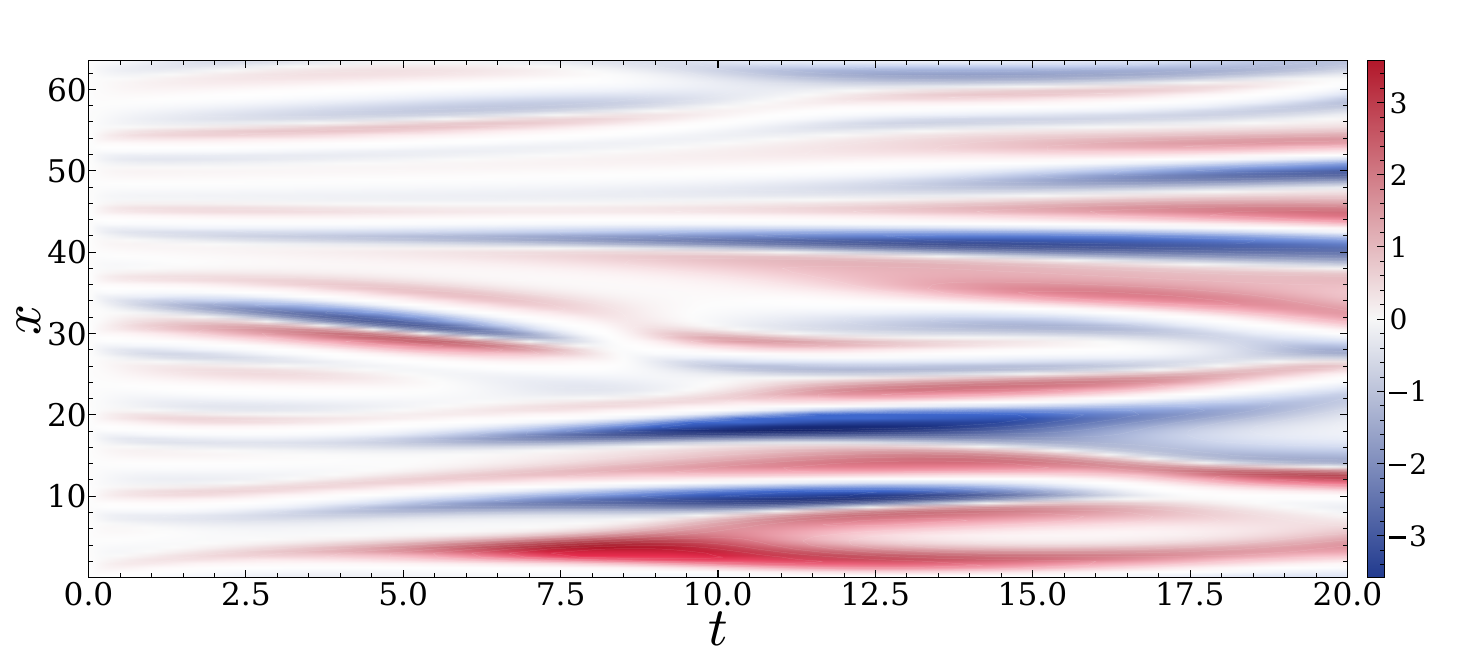}
    \caption{Left: the error field of the rollout from the KANDy models. Right: the error field of the space time from Equation~\eqref{eq:ks_discovered}.}
    \label{fig:equation-spacetime}
\end{figure*}

Figure~\ref{fig:equation-spacetime} shows the pointwise error field $u_{\text{KANDy}}(x,t) - u_{\text{true}(x,t)}$ of the KANDy rollout versus the integrated space--time, numerically integrated on a periodic domain using a Fourier pseudospectral method with exponential time differencing (left), and also, the error field of the discovered equations after removing the constant forcing term and rescaling the variables to match the normalization (right). The initial condition was set to match that used for the ground-truth data, and a $200$-sample burn-in interval was discarded to ensure the solution lies on the attractor. The resulting space-time field exhibits the characteristic slanted-stripe patterns and spatiotemporal chaos associated with the Kuramoto-Sivashinsky dynamics, demonstrating close qualitative agreement with the ground-truth evolution.

\subsubsection*{Inviscid Burgers}

Our aim is to examine the behavior of the solution under several chosen initial and boundary conditions and to assess the performance of the implemented numerical scheme, giving particular attention to the development of nonlinear wave steepening, the formation of shocks, and the ability of the method to capture discontinuities without introducing spurious oscillations, without relying on network depth. We first consider the evolution of the solution over time and compare the wave speed, shock location, and conservation properties (where appropriate) with the analytical solution, and also examine the effects of grid resolution and time-stepping parameters on the quality of the solution.

\begin{figure}[!htbp]
    \centering
    \includegraphics[width=1.0\linewidth]{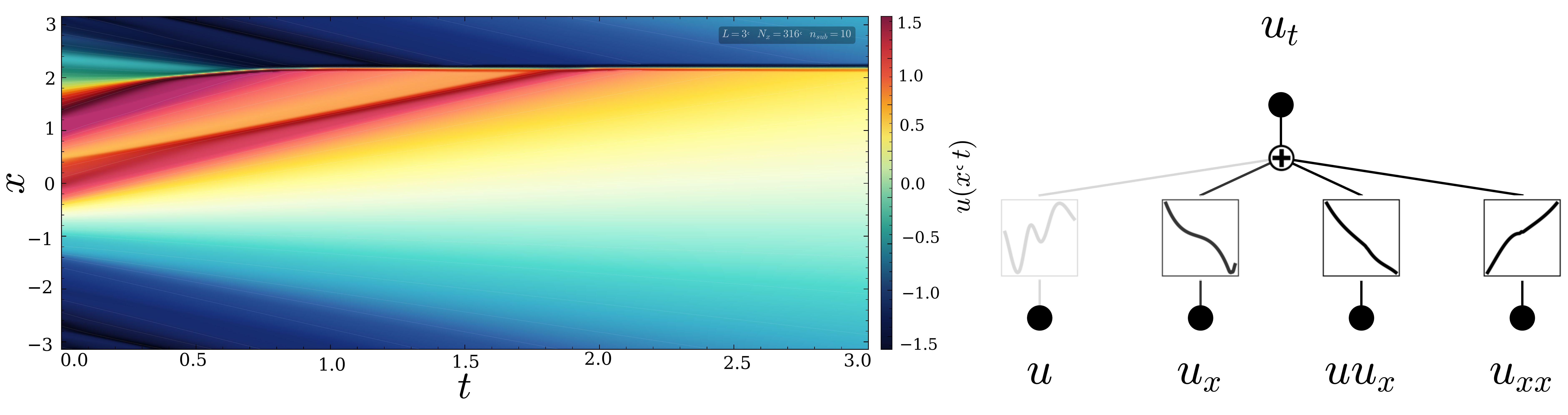}
    \caption{Left: The KANDy rollout of the inviscid-Burgers equation with random Fourier mode initial conditions. Right: the learned Inviscid Burger's equation trained on data over the shock.}
    \label{fig:burgers_side_by_side}
\end{figure}

Figure~\ref{fig:burgers_side_by_side} shows the learned Inviscid Burger's equation trained on data over the shock, where the KANDy model with a few lifted features, such as the advection term, which is absorbing the shock starting from random Fourier mode initial conditions, producing extreme shocks over which sparse-regression struggles. 

\begin{equation}
u_0(x) = \sum_{k=1}^{K}
\xi_k k^{-p} \sin(kx + \phi_k),
\quad
\xi_k \sim \mathcal{N}(0,1),
\quad
\phi_k \sim \mathrm{Uniform}(0,2\pi).
\label{eq:fourier_ics}
\end{equation}

Equation~\eqref{eq:fourier_ics} shows the choice of random Fourier mode initial conditions generated in numpy with a random seed of 42 for reproducibility, with 20 modes selected. Lifted inputs to the KANDy model were $u$, $u_x$, $uu_x$, $u_{xx}$. The rollout integration horizon was 5, and the integration was performed. High-resolution reference solutions were generated using a method-of-lines discretization with a Rusanov (local Lax--Friedrichs) flux and adaptive Runge--Kutta time integration. This numerical scheme robustly captures shock formation and propagation without introducing spurious oscillations. The training balanced integration with the model-guided MSE, using a weight of 0.5 (half MSE and half integral loss). Integration loss and MSE loss on the train and test trajectories, both reduced by an order of magnitude during training. The Burgers' equation used for this experiment, $u_t + u u_x = \nu u_{xx}$, also contained a viscosity term $\nu$, which KANDy reproduces in the governing equations. 

\begin{equation}
 u_t \;=\;- 1.158\,u\,u_x  + 1.566 u_{xx}
\label{eq:discovered_inviscid_burgers_random}
\end{equation}

The robust symbolic fitting used a 0.8 threshold for both complexity weight and $R^2$-threshold. Equation~\eqref{eq:discovered_inviscid_burgers_random} shows the governing equations estimated for this system. The estimated $\nu$ term was initially implausible from the physical standpoint until it was renormalized by the standard deviation introduced during the data-generating process $\sigma^2=179.76$, which, after rescaling, set $\mu=0.002$ whereas the viscosity term introduced to the model was $0.001$. For the $u\,u_{xx}$ term $\sigma^2=1.69$, which after rescaling put this coefficient at $-0.70$ where the true coefficient was $-1$, meaning that this term had a relative error of 30\%. 

\begin{figure}[!htbp]
    \centering
    \includegraphics[width=1.0\linewidth]{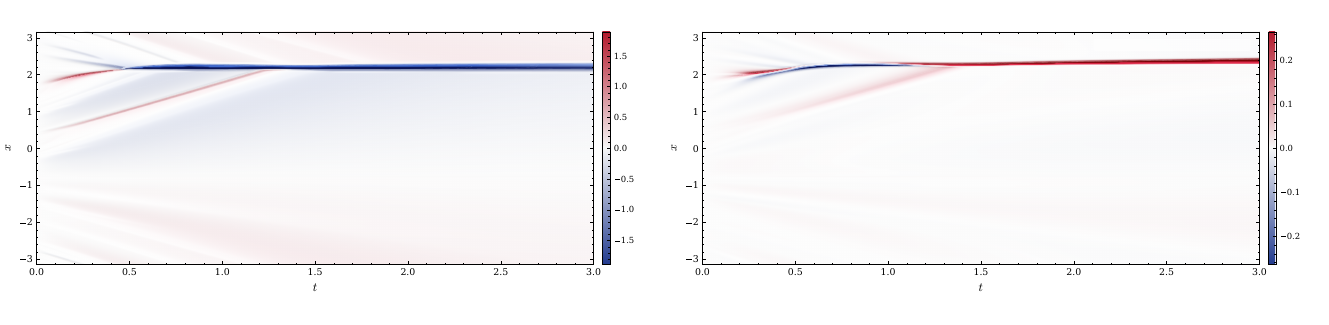}
    \caption{Left: The error field of KANDy rollout of the inviscid-Burgers equation with random Fourier mode initial conditions. Right: the learned Inviscid Burger's equation trained on data over the shock.}
    \label{fig:burgers_error_field}
\end{figure}

Figure~\ref{fig:burgers_error_field} shows the error field for both the learned KANDy model (left) and the discovered governing equations (right). Both error fields show the desired wave and shock formation that is quintessential of the dynamics of Inviscid Burgers'. The governing equations along the shock formation show a higher error in red, while the model learned a lower error (blue) in line with the abrupt shock both precisely along the singularity.

\begin{equation}
 u_t \;=\;- 1.158\,u\,u_x  - 0.02
\label{eq:discovered_inviscid_burgers}
\end{equation}

Next, we consider the one-dimensional inviscid Burgers equation on a periodic domain, $u_t + u u_x = \nu u_{xx}$, initialized with a smooth sinusoidal profile.  Equation~\eqref{eq:discovered_inviscid_burgers} shows the governing equations discovered with KANDy on the inviscid Burger's equation integrated over the shock.
However, a tight fit for this particular spline is the hyperbolic secant squared $\mbox{sech}^2(x)$ with $R^2=0.99$, demonstrating that KANDy adequately learned the shock of this term. However, the global function fitting reduces this term to a linear term after weighting for complexity and global consistency, thereby successfully learning the governing equation. As expected, characteristics intersect at approximately $t \approx 1$, after which the solution develops a sharp discontinuity. The reference solution spans both pre-shock and post-shock regimes, providing a stringent benchmark for data-driven discovery under non-smooth dynamics.

The training sample was explicitly taken over both smooth and post-shock intervals. Derivatives were approximated using forward differences, yielding supervision signals with large, localized gradients near the shock. This setting is deliberately challenging, as many existing PDE discovery methods degrade significantly in the presence of discontinuities and sparse temporal sampling. Despite these challenges, KANDy remained numerically stable throughout training, rollout, and equation discovery, indicating robustness to both temporal sparsity and shock-induced nonlinearity. KANDy was trained to predict the local time derivative $u_t$ at each spatial grid point. The input feature library was deliberately minimal and physically structured, consisting of the local state $u$, and the spatial derivative of the flux, $\partial_x\!\left(\tfrac{1}{2}u^2\right)$. KANDy achieved low training and test root-mean-square errors on the supervised derivative data, indicating that the local temporal dynamics were learned accurately despite the presence of shocks.

To ensure that the learned local dynamics were globally consistent in time, a differentiable rollout loss was incorporated. Short trajectory windows were sampled from the reference solution, and the learned model was integrated forward in time using a fourth-order Runge--Kutta scheme. The rollout loss penalized deviations between predicted and true states over these short horizons. 

Incorporating this rollout loss significantly improved stability and long-term accuracy. Models trained without rollout supervision exhibited rapid error accumulation and failed to reproduce post-shock structure. In contrast, rollout-regularized models remained stable across the entire temporal domain.

Extending this analysis to a more complicated case, we applied KANDy to inviscid Burgers' with random Fourier mode initial conditions. Figure~\ref{fig:burgers_side_by_side} shows the rollout of the KANDy model as well as the model architecture taking in the lifted $u$, $u_x$, $uu_x$, $u_{xx}$ starting from the initial condition $u(x,0)=\sin(x)$, the learned dynamics were integrated forward across the full simulation interval. The resulting spatiotemporal solution closely matched the ground-truth solution, including correct shock formation time, shock location, and propagation speed.

At the final time, the predicted solution aligned closely with the reference solution across the entire spatial domain, with only minor discrepancies localized near the shock front. Notably, no artificial diffusion or oscillatory artifacts were observed, despite the absence of explicit shock-capturing terms in the learned model.

To analyze the internal structure of the learned model, edge activations within the KAN were extracted and examined. One dominant edge exhibited a strong linear relationship between its input and output activations, with a coefficient of determination $R^2$ close to unity. Linear regression revealed that this edge effectively encoded an affine transformation of the flux gradient term.

Alternative nonlinear models, including sinusoidal and composite sine--linear functions, produced substantially lower $R^2$ values, confirming that the learned relationship was fundamentally linear in nature.

Additional edge activations exhibited highly localized spatial structures centered near the shock location. These activations were well-fit by functions of the form $\mathrm{sech}^2\!\left(\frac{x - x_0}{\ell}\right)\tanh\!\left(\frac{x - x_0}{\ell}\right)$ which are characteristic of shock-layer derivatives in Burgers-type equations. Nonlinear least-squares fitting yielded high coefficient of determination values, indicating that the KAN internally represented a physically meaningful shock geometry.

Correlation and symmetry analyses further confirmed strong agreement between the learned activations and analytically motivated shock profiles.
Following training, symbolic regression was applied to the learned KAN representation. The dominant symbolic expression recovered by the model corresponded directly to the inviscid Burgers equation, $u_t \approx -\partial_x\!\left(\tfrac{1}{2}u^2\right)$, up to an affine scaling factor. Importantly, this expression emerged without explicit enforcement during training and remained valid across both smooth and shock-dominated regimes.


\begin{table}[!htpb]
\centering
\begin{tabular}{l c >{$}l<{$}}
\hline
\textbf{Method} &\textbf{NRMSE} & \multicolumn{1}{c}{\textbf{Discovered Equation}} \\
\hline
\textbf{KANDy} & 0.049 &
\begin{aligned}[t]
u_t &= -1.084\,u\,u_x - 0.026
\end{aligned} \\
OLS & 0.150 &
\begin{aligned}[t]
u_t &= -1.083\,u\,u_x + 0.063\,u - 0.082\,u_x \\
    &\quad + 0.020\,u_{xx} + 0.024
\end{aligned} \\
LASSO & 0.150 &
\begin{aligned}[t]
u_t &= -1.081\,u\,u_x + 0.062\,u - 0.082\,u_x \\
    &\quad + 0.020\,u_{xx} + 0.024
\end{aligned} \\
PDE-FIND FD & 1.562 &
\begin{aligned}[t]
u_t &= 12.272\,u_x
\end{aligned} \\
PDE-FIND SmoothedFD & 1.562 &
\begin{aligned}[t]
u_t &= 12.272\,u_x
\end{aligned} \\
\hline
\end{tabular}
\caption{Head-to-head comparison of KANDy and baseline methods on inviscid Burgers with random Fourier-mode initial conditions. The true equation is $u_t = -u\,u_x$. Coefficients are rounded and terms with $|\text{coeff}| \le 0.01$ are omitted.}
\label{tab:burgers_baseline_comparison}
\end{table}

Table~\ref{tab:burgers_baseline_comparison} shows the comparison is carried out on a matched inviscid Burgers benchmark designed to mirror our initial inviscid Burgers experiment, except performed using AI agents to reduce human error in fine-tuning each of the methods. Both PDE-Find tests only had one active term, while both sparse regressions had 5. Only KANDy has the single active flux term and obtained the lower NRMSE of $0.048$. The state is discretized on a one-dimensional periodic grid with $N_x=128$, with a spatial domain of $[-\pi,\pi)$ with no duplicated endpoints. Initial conditions are drawn from a random Fourier ensemble with $K=10$ modes, random seed $0$, and mode amplitudes scaled as $\sigma_k=k^{-1.5}$, which biases the initialization toward smoother low-frequency structure while still allowing enough complexity to form shocks. The trajectory was generated over the time interval $t\in[0,2]$ with step size $dt=0.004$ with a first-order Rusanov flux solver integrated with RK45 from SciPy using tolerances $rtol=1e-6$ and $atol=1e-8$. Those settings matter because they define both the spatial resolution and the difficulty of the identification problem: the data are shock-forming, numerically stable, and sampled densely enough that derivative quality is the dominant bottleneck rather than sampling scarcity.

For KANDy, the model setup is intentionally aligned with the problem's shock structure. Spatial derivatives are computed with a TVD minmod scheme and temporal derivatives with a forward difference, $(U^{n+1}-U^n)/dt$, rather than central differencing. The feature library is very small and physics-motivated: $[u,\;u_x,\;u\,u_x,\;u_{xx}]$. This gives the model direct access to the true nonlinear transport term while still allowing it to test whether lower-order nuisance terms such as $u$, $u_x$, or $u_{xx}$ are needed. These four features are passed into a KAN with architecture $[4,1]$, meaning four inputs and one scalar output for $u_t$. The spline grid is set to $7$ and the spline order to $k=3$, giving the network enough flexibility to represent nonlinear relationships without making the symbolic form excessively hard to interpret. 300 training steps were used. In this setup, sparsity is not imposed by a separate thresholding stage; instead, it emerges from learned edge activations, allowing inactive inputs to effectively collapse during training. This is why KANDy can retain only the dominant $u\,u_x$ contribution while suppressing the others.

The PDE-FIND baseline is built using PySINDy’s PDELibrary together with a derivative order $3$, and STLSQ as the sparse optimizer. The degree $2$ polynomial library supports combinations up to quadratic order, which is, in principle, sufficient for the Burgers nonlinearity. The derivative order of 3 means the library can include spatial derivatives up to third order, though in the reported comparisons the key terms of interest are $u_x$ and $u_{xx}$. STLSQ then performs sequential thresholded least squares to prune coefficients and recover a sparse equation. Three temporal differentiation choices are tested in the PySINDy stack: finite difference, smoothed finite difference, and Savitzky--Golay derivative estimates. However, these affect only the estimate of $u_t$, not the construction of the spatial derivative features in PDELibrary. That distinction is central to the failure mode: even with alternative temporal differentiation, the spatial feature matrix remains corrupted because the library uses its own internal finite-difference machinery rather than the shock-aware TVD derivative used by KANDy.


Finally, the symbolic extraction procedure for KANDy introduces a second layer of “model setup” that is separate from training but essential for producing a readable discovered equation. A deep copy of the trained KAN is made because symbolic fitting mutates the learned spline structure. Activation caching is re-enabled, and then a forward pass on a batch of data populates the caches needed for symbolic fitting. Symbolic conversion is then run with the KANDy library with the \verb|robust_auto_symbolic| function using the library $x$, $x^2$, $x^3$, with $R^2=0.80$, $C=0.80$, $K_{\text{top}}=8$ as the default inputs. These hyperparameters balance fidelity and simplicity: only edges with a sufficiently good symbolic fit are retained, preference is given to simpler expressions, and only the most important edges are considered. After extraction, the expression is expanded, and coefficients below a tolerance of $0.01$ are removed. That rounding step is important because symbolic fitting can create algebraically complicated but nearly cancelling artifacts; expanding and thresholding ensure that the final reported PDE reflects the actual learned structure rather than spurious symbolic clutter.

The AI agents used for this experiment are ``Physics-Informed" in that they access pre-tested written software that implements the KANDy algorithm. The complete code implementation, as well as agent memory for this experiment, is found in GitHub.\footnote{\href{https://github.com/Center-For-Complex-Systems-Science/kandy}{https://github.com/Center-For-Complex-Systems-Science/kandy}}

\section*{Discussion}

This work introduces KANDy, a data-driven framework for dynamical system discovery and forecasting built around zero-depth KANs with library inputs augmented with library terms and optional integrator fine-tuning.  We find that, although KANs were developed to enable deep learning (network depths $\geq 2$), network depth hinders the discovery of governing equations, slightly degrades or does not improve learning vector-field approximations, whereas, juxtaposed with KANDy, it preserves invariant structure across a range of discrete, continuous, chaotic, and geometric settings in terms of both quotient and network topologies (e.g. learns fractal and continuous but undifferentiable patterns from data and reconstructs Kuramoto Oscillator Networks). Shifting from depth to expressivity, Kolmogorov--Arnold structured representations yield models that are simultaneously accurate, stable under rollout, and more directly interpretable for equation discovery.

A conventional deep-learning perspective is that depth is required to capture complex dynamics via hierarchical composition. In contrast, our results suggest that for many dynamical systems, the relevant compositional structure is already supplied by the physics (e.g., polynomial and bilinear interactions, symmetries, and differential operators), and the learning problem reduces to identifying a stable nonlinear map consistent with that structure. Figure~\ref{fig:zero-depth} illustrates this principle in the Lorenz setting: a zero-depth KAN equipped with a small dictionary of nonlinear terms (e.g., products such as $xy,xz,yz$) is sufficient to approximate the vector field and, when regularized toward simplicity, to support governing-equation recovery. In this sense, KANDy reallocates model capacity away from learning deep latent features and toward learning calibrated nonlinear responses in an identifiable functional parameterization.

A central challenge in chaotic prediction is that long-horizon phase accuracy is fundamentally limited by the largest Lyapunov exponent, even for perfect models. Consequently, a meaningful evaluation must combine (i) short-horizon accuracy, (ii) error-growth behavior relative to the Lyapunov time, and (iii) preservation of invariant-set geometry. Figures~\ref{fig:lorenz} (Lorenz attractor reconstruction) and the out-of-distribution rollout figure (Fig.~\ref{fig:lorenz}) show that the learned dynamics reproduce the global geometry and folding structure of the Lorenz attractor, and that trajectories initialized off-distribution but within the basin of attraction converge back onto the attractor after a transient.

Quantitatively, the normalized rollout error exhibits a transition consistent with the theoretical predictability horizon: Fig.~\ref{fig:lyapunov_curve} shows rapid error growth on a timescale aligned with the Lyapunov time, after which phase decorrelation is expected. Importantly, this degradation in phase does not imply a collapse of the learned dynamics; rather, it reflects faithful reproduction of the system's expansion rates and instability structure. The combined qualitative and quantitative evidence supports the interpretation that KANDy learns a correct chaotic mechanism up to the intrinsic predictability limit and preserves long-horizon geometric structure beyond it.

Many sequence models trained with one-step (teacher-forced) losses perform well locally but drift rapidly when deployed autoregressively due to compounding distributional shift. Our autoregressive integration regularization directly targets this mismatch by optimizing parameters through rollout behavior. The practical effect is visible in the Lorenz rollout experiments (Fig.~\ref{fig:lorenz} and the subsequent rollout figure): the learned map remains dynamically consistent on its own generated trajectories and returns to the attractor from nearby off-training initial conditions. This supports the broader claim that stable chaos-aware forecasting benefits from training objectives that explicitly include a multi-step rollout structure rather than relying solely on one-step regression. Applying this paradigm as a replacement for the least-squares estimate used in sparse regression for equation discovery greatly improves the robustness of the discovered models.

For spatiotemporal chaos, learning the full PDE flow map end-to-end can be data-hungry and unstable, while pure numerical integration can suffer from discretization error and unresolved effects. We therefore adopt a hybrid approach for the Kuramoto--Sivashinsky equation. Adding advection, hyperdiffusion, and other physics-informed width-building terms to the KANDy zero-depth wide network improves the autoregressive rollout. 

Beyond forecasting, a distinguishing outcome of KANDy is its ability to recover invariant and topological structure from data. The Hopf fibration experiment (Fig.~\ref{fig:hopf_main}) demonstrates that a trained KAN can learn a quotient map that collapses entire group orbits to points on the quotient space, providing a concrete example of data-driven symmetry reduction. In dynamical contexts, such symmetry-aware representations can simplify the effective dynamics and improve interpretability.

KANDy is not intended to extrapolate arbitrarily outside the training support or outside the basin of attraction. The Lorenz experiments highlight that while off-distribution initial conditions within the basin can converge to the attractor (rollout figure following Fig.~\ref{fig:lorenz}), sufficiently distant initializations can diverge, consistent with chaotic sensitivity and data-driven modeling limits. Additionally, library design introduces a trade-off: overly rich dictionaries can lead to non-identifiability (multiple explanations fit the data), while overly sparse dictionaries can prevent recovery of the true governing form. Finally, symbolic extraction and clean governing-equation recovery may degrade under substantial noise or insufficient sampling, and scaling residual corrections to very high-dimensional PDE states may benefit from additional parameter sharing (e.g., spectral or convolutional structure).

These results suggest a synthesis between sparse-dictionary discovery, chaos-aware rollout training, and invariant learning. Promising directions include: (i) complexity-controlled discovery via explicit sparsity and penalties on spline or library complexity; (ii) enforcing known symmetries (e.g., translation in KS) or jointly learning invariances through quotient objectives; (iii) evaluating learned models beyond phase accuracy using invariant measures (spectra, correlation dimension, Lyapunov spectrum estimates); and (iv) extending residual-learning KANDy to multiscale PDE settings where unresolved physics can be treated as a learnable correction.

Finally, KANDy can recover governing equations even when traditional sparse regression may fail, e.g., in Kuramoto Oscillators, the Ikeda optical cavity map, and Holling Type II systems. KANDy not only estimates the governing equations but also maps their structure back to the original network, providing a potential future research direction for recovering functional network structures purely from the phases of coupled oscillators. 

KANDy reframes equation discovery as the search for an appropriate lifted coordinate system in which nonlinear dynamics become structurally simple. The lift map transforms the system's states into a higher-dimensional, potentially infinite, space where the dynamics are approximately linear. The effectiveness of this approach across chaotic ODEs, PDEs, and topological quotient maps indicates that the primary challenge lies in representation rather than depth. This representation diverges from the analogy of the Kolmogorov-Arnold Representation theorem. Although not explicitly stated to avoid confusion with methodologies that employ the Koopman formalism, the lift map aligns with a finite-dimensional observable embedding consistent with Koopman theory. KANDy integrates sparse regression, Koopman theory, and Kolmogorov–Arnold networks into a unified framework for modeling dynamical systems. The principled synthesis of Kolmogorov-Arnold Networks and sparse regression (KANDy) addresses the loss of multivariate structure that occurs when decomposition is restricted to univariate inner functions.

\section*{Acknowledgments}
KS, JF and EB gratefully acknowledge this work was funded by the Army Research Office 
under award no.~W911NF2310393.

The KANDy software is offered freely for use and is an Agentic AI software system built on Claude Code with models accessed between August, 2025 and March 2026 using the Claude Opus 4.6 and Claude Sonnet models. All agent personality prompts are available for review, and only call tools are available from a Python API developed from rigorous scientific code and experiments. AI Agents only modified code based on rigorous experimentation, and all generated agent code is constrained to a single Python script for review by the authors. AI tools were used to polish the text. Complete code implementation with agent memory is found in GitHub.\footnote{\href{https://github.com/Center-For-Complex-Systems-Science/kandy}{https://github.com/Center-For-Complex-Systems-Science/kandy}}

\bibliographystyle{unsrtnat}
\bibliography{references}

\appendix

\section*{Proofs}

\begin{theorem}\label{thm:single_neuron} 
There does not exist $h,u,v:[0,1]\mapsto \mathbb{R}$ such that $f(x,y) = xy = h(u(x) + v(y))$ for all $x,y\in I^2$ (the unit square).
\end{theorem}

\begin{proof}
By way of contradiction, suppose  $h,u,v$ exist.
Let $y=0$ so that we have 
$$0=x\times 0= h(u(x) + v(0)).$$

Since $u:I\mapsto\mathbb{R}$ is continuous we have $u(I)=[a,b]$ for some $a$ and $b$ with $a<b$ and let $v(0)=c$.

Then $h(t)=0$ for all $t\in[a+c, b+c]$.
 
Next, we show a small change in $y$ forces a spurious zero for $y>0$

If $a=b$, then $u$ is constant, so $h(u(x)+v(y))$
depends only on$y$, but $xy$ depends on $x$; contradiction. Hence $a<b$.

Since $v$ is continuous at $0$, choose $y_*>0$ and define $d$ such that $d=v(y_*) -c$, and this satisfied $|d| < b -a$.

But then $a-d\in [a,b]$, so there exists $x_* \in I$ such that $u(x_*) = a-d$.

Then we have 
$$u(x_*) + v(y_*) = (a -d) + (c + d) = a + c.$$

But  $a + c \in [a+c, b+c]$, so $h(u(x_*) + v(y_*))=0$.

By construction $h(u(x_*) + v(y_*))=x_*y_*$ and since $y_*>0$, it must be the case that $x_*=0$.

Recall $x_*$ was chosen so that $u(x_*)=a-d$, then $u(0)=a-d$ and because this holds for arbitrarily small $y$, e.g., $d=v(y)-c$ we get $u(0)$ takes on multiple values which is a contradiction.

Since for arbitrarily small $y>0$ we obtain $u(0)=a-(v(y)-c)$ and the right-hand side varies with $y$ by continuity of $v$, this forces $u(0)$ to take multiple distinct values, contradicting that $u$ is a well-defined function. $\square$
\end{proof}

\begin{theorem}\label{thm:deep_kan_xy}
Let $\mathrm{KAN}_\theta:\mathbb{R}^3\to\mathbb{R}^3$ be a KAN with architecture
$$
(3,1,1,\dots,1,3),
$$
so that every hidden layer has width $1$. Assume each node is of KAN type: it is a sum of univariate functions of its inputs. Then $\mathrm{KAN}_\theta$ cannot represent the Lorenz component
$$
f_3(x,y,z)=xy-\beta z
$$
on $[0,1]^3$. Consequently, arbitrarily deep KANs with no additional hidden width cannot represent the bilinear nonlinearities in the Lorenz system.
\end{theorem}

\begin{proof}
Because the first hidden layer has width $1$, its output has the form
$$
s_1(x,y,z)=u(x)+v(y)+w(z)
$$
for some continuous univariate functions $u,v,w$.

Since each subsequent hidden layer also has width $1$, its output is obtained by applying a continuous univariate function to the previous scalar output. Thus, after any finite number of hidden layers, the final hidden representation has the form
$$
s_L(x,y,z)=\Phi(u(x)+v(y)+w(z))
$$
for some continuous univariate function $\Phi$.

Each output coordinate is then a continuous univariate function of $s_L$, so in particular the third output coordinate must be of the form
$$
\mathrm{KAN}_{\theta,3}(x,y,z)=H(u(x)+v(y)+w(z))
$$
for some continuous univariate function $H$ (after absorbing $\Phi$ into $H$).

Suppose, for contradiction, that this equals the Lorenz component:
$$
xy-\beta z = H(u(x)+v(y)+w(z))
\qquad \text{for all } (x,y,z)\in [0,1]^3.
$$
Set $z=0$. Then
$$
xy = H(u(x)+v(y)+w(0)).
$$
Define $\widetilde v(y)=v(y)+w(0)$. Then
$$
xy = H(u(x)+\widetilde v(y))
\qquad \text{for all } (x,y)\in [0,1]^2.
$$
This contradicts Theorem~\ref{thm:single_neuron}, which states that there do not exist continuous functions $H,u,\widetilde v:[0,1]\to\mathbb{R}$ such that
$$
xy=H(u(x)+\widetilde v(y))
$$
for all $(x,y)\in [0,1]^2$.

Therefore, the third coordinate of a width-$1$ deep KAN cannot represent $xy-\beta z$. Hence, arbitrarily deep KANs with no additional width cannot represent the bilinear terms in the Lorenz system.
\end{proof}

\begin{proposition}\label{prop:quadratic_obstruction}
Let $V$ be the linear span of quadratic hidden features
$$
\ell_j(x,y,z)^2,\qquad \ell_j(x,y,z)=a_jx+b_jy+c_jz,\qquad j=1,\dots,m.
$$
If $xy,xz\in V$, then $V$ must also contain nontrivial linear combinations involving diagonal terms $x^2,y^2,z^2$ unless the coefficients of the generators satisfy a system of cancellation identities. In particular, the realization of the Lorenz nonlinearities $xy$ and $xz$ from quadratic additive channels is not termwise sparse: cross terms are necessarily produced together with diagonal terms at the feature level and only disappear, if at all, after cancellation across channels.
\end{proposition}

\begin{lemma}\label{lem:sign_plane}
Let
$$
\Sigma=\{(1,1,1),\ (1,-1,-1),\ (-1,1,-1),\ (-1,-1,1)\}\subset \mathbb R^3.
$$
Then no two-dimensional subspace generated by differences of vectors in $\Sigma$ contains both $(1,0,0)$ and $(0,1,0)$.
\end{lemma}

\begin{proof}
Each vector in $\Sigma$ has a coordinate product equal to $1$. The differences of such vectors are, up to sign,
$$
(0,2,2),\ (2,0,2),\ (2,2,0),\ (0,2,-2),\ (2,0,-2),\ (2,-2,0).
$$
Thus, every difference has either one zero entry and two nonzero entries or two entries of opposite sign. Any two-dimensional subspace generated by such differences is contained in one of the coordinate-sum planes $u_1+u_2=0$, $u_1+u_3=0$, $u_2+u_3=0$, or one of their sign variants. Neither $(1,0,0)$ nor $(0,1,0)$ lies simultaneously in any such plane. Hence, no such subspace contains both vectors.
\end{proof}

\begin{theorem}\label{thm:lorenz_width3_obstruction}
Consider a KAN with architecture $[3,3,3,3]$, and suppose each spline is restricted to the monomial dictionary
$$
\psi(t)\in \mathrm{span}\{1,t,t^2,\dots,t^N\}, \qquad N\ge 2.
$$
Assume the quadratic nonlinearities in the output are generated from three hidden quadratic channels of the form $(v_i^\top x)^2,$ for $i=1,2,3$, where $x=(x,y,z)^\top\in\mathbb R^3$ and $v_i\in\mathbb R^3$.

Then the span of these three quadratic channels cannot contain both bilinear forms $xy$ and $xz$.
Consequently, a width-$3$ monomial KAN cannot realize the full quadratic nonlinear part of the Lorenz vector field
$$
(\rho x-xz-y,\; xy-\beta z)
$$
from only three quadratic hidden channels.
\end{theorem}

\begin{proof}
Any homogeneous quadratic polynomial in $(x,y,z)$ can be written uniquely as $q(x)=x^\top A x$, where $A\in \mathrm{Sym}_3$ is a real symmetric $3\times 3$ matrix. The bilinear monomials $xy$ and $xz$ correspond to the symmetric matrices
$$
A_{xy}
=
\frac12
\begin{pmatrix}
0&1&0\\
1&0&0\\
0&0&0
\end{pmatrix},
\qquad
A_{xz}
=
\frac12
\begin{pmatrix}
0&0&1\\
0&0&0\\
1&0&0
\end{pmatrix},
$$
since
$$
x^\top A_{xy}x = xy,
\qquad
x^\top A_{xz}x = xz.
$$

Now suppose, for contradiction, that both $A_{xy}$ and $A_{xz}$ lie in the span
$$
S:=\mathrm{span}\{v_1v_1^\top,\; v_2v_2^\top,\; v_3v_3^\top\}\subset \mathrm{Sym}_3.
$$
Each matrix $v_iv_i^\top$ has rank one.

Consider the diagonal map
$$
\mathrm{diag}:\mathrm{Sym}_3\to \mathbb R^3,
\qquad
\mathrm{diag}(A)=(A_{11},A_{22},A_{33}).
$$
Since both $A_{xy}$ and $A_{xz}$ have zero diagonal, we have
$$
A_{xy},A_{xz}\in \ker(\mathrm{diag}|_S).
$$
These two matrices are linearly independent, so
$$
\dim \ker(\mathrm{diag}|_S)\ge 2.
$$
Because $\dim S\le 3$, rank-nullity implies
$$
\dim \mathrm{diag}(S)\le 1.
$$

Write
$$
v_i=(a_i,b_i,c_i)^\top.
$$
Then
$$
\mathrm{diag}(v_iv_i^\top)=(a_i^2,b_i^2,c_i^2).
$$
Since $\mathrm{diag}(S)$ is one-dimensional, the three vectors
$$
(a_i^2,b_i^2,c_i^2), \qquad i=1,2,3,
$$
must all be collinear in $\mathbb R^3$. Hence, there exists a fixed vector
$$
(r_1^2,r_2^2,r_3^2)
$$
and scalars $t_i^2\ge 0$ such that
$$
(a_i^2,b_i^2,c_i^2)=t_i^2(r_1^2,r_2^2,r_3^2).
$$
Therefore, up to signs,
$$
v_i=t_i(\varepsilon_i r_1,\eta_i r_2,\theta_i r_3),
\qquad
\varepsilon_i,\eta_i,\theta_i\in\{\pm 1\}.
$$

It follows that
$$
v_iv_i^\top
=
t_i^2
\begin{pmatrix}
r_1^2 & \varepsilon_i\eta_i r_1r_2 & \varepsilon_i\theta_i r_1r_3\\
\varepsilon_i\eta_i r_1r_2 & r_2^2 & \eta_i\theta_i r_2r_3\\
\varepsilon_i\theta_i r_1r_3 & \eta_i\theta_i r_2r_3 & r_3^2
\end{pmatrix}.
$$

Now let
$$
B=\sum_{i=1}^3 \lambda_i v_iv_i^\top \in S.
$$
Its diagonal entries are
$$
\mathrm{diag}(B)=
\Big(\sum_{i=1}^3 \lambda_i t_i^2\Big)(r_1^2,r_2^2,r_3^2).
$$
Thus $B$ has zero diagonal if and only if
$$
\sum_{i=1}^3 \lambda_i t_i^2=0.
$$
For such a $B$, the off-diagonal entries are
$$
(B_{12},B_{13},B_{23})
=
\left(
r_1r_2\sum_{i=1}^3 \lambda_i t_i^2 \varepsilon_i\eta_i,\;
r_1r_3\sum_{i=1}^3 \lambda_i t_i^2 \varepsilon_i\theta_i,\;
r_2r_3\sum_{i=1}^3 \lambda_i t_i^2 \eta_i\theta_i
\right).
$$

Now note that each sign triple
$$
(\varepsilon_i\eta_i,\varepsilon_i\theta_i,\eta_i\theta_i)
$$
must belong to the set
$$
\{(1,1,1),\ (1,-1,-1),\ (-1,1,-1),\ (-1,-1,1)\},
$$
because the product of the three components is always $1$. Hence, the space of zero-diagonal matrices in $S$ is generated by differences of these sign patterns. But any such space has dimension at most $2$, and its off-diagonal vectors lie in the plane generated by differences of the above four sign vectors. A direct inspection shows that this plane cannot contain both
$$
(1,0,0)
\qquad\text{and}\qquad
(0,1,0),
$$
which are the off-diagonal signatures corresponding to $A_{xy}$ and $A_{xz}$, respectively. Therefore, $S$ cannot contain both $A_{xy}$ and $A_{xz}$, contradicting the assumption. We conclude that three quadratic channels of the form $(v_i^\top x)^2$ cannot simultaneously generate the two independent bilinear forms $xy$ and $xz$. Hence, a width-$3$ monomial KAN cannot realize the full quadratic nonlinear part of the Lorenz system from only three quadratic hidden channels.
\end{proof}

\begin{corollary}
A width-$3$ monomial KAN may represent an individual bilinear term such as $xy$ by polarization, but it cannot generate the full pair of Lorenz bilinear terms $(xy,xz)$ from only three quadratic hidden features. Thus, depth alone does not overcome the algebraic obstruction at fixed width $3$.
\end{corollary}

\begin{theorem}\label{thm:inductive_width3}
Consider a deep KAN with architecture $[3,3,\dots,3]$ and monomial dictionary
$$
\psi(t)\in \mathrm{span}\{1,t,t^2,\dots,t^N\},\qquad N\ge 2.
$$
For each layer $\ell$, let $Q_\ell$ denote the vector space of homogeneous quadratic forms appearing in the three hidden coordinates at layer $\ell$. Then $\dim Q_\ell \le 3$, and in fact $Q_\ell$ is spanned by at most three rank-1 quadratic forms.
\end{theorem}

\begin{proof}
We proceed by induction on the layer index $\ell$. For the first hidden layer, each coordinate is of the form
$$
h_j^{(1)}=\sum_{i=1}^3 \psi_{ij}(x_i),
$$
where $x_1=x$, $x_2=y$, $x_3=z$. The quadratic part of $h_j^{(1)}$ is therefore a linear combination of $x^2,y^2,z^2$, so $Q_1$ is spanned by at most three rank-1 quadratic forms, namely
$$
x^2,\ y^2,\ z^2.
$$
Hence, $\dim Q_1\le 3$. Assume now that at layer $\ell$, the quadratic space $Q_\ell$ is spanned by at most three rank-1 quadratic forms. Consider layer $\ell+1$. Each coordinate at layer $\ell+1$ is obtained by applying monomial splines to additive combinations of the three coordinates at layer $\ell$. The quadratic contribution at layer $\ell+1$ can only arise from squaring the linear parts of these additive combinations, because higher monomials contribute terms of degree at least $3$.

Let $u_1,u_2,u_3$ denote the linear parts of the three layer-$\ell$ coordinates. Then each new quadratic contribution has the form
$$
(a_1u_1+a_2u_2+a_3u_3)^2,
$$
which is a rank-1 quadratic form in the three-dimensional linear space spanned by $u_1,u_2,u_3$.

Since layer $\ell+1$ has width $3$, there are at most three such output coordinates, so the entire quadratic space $Q_{\ell+1}$ is spanned by at most three rank-1 quadratic forms. Therefore
$$
\dim Q_{\ell+1}\le 3.
$$

This completes the induction.
\end{proof}
\newpage
\input{ex_supplement}

\end{document}

%% file: ex_shared.tex
\usepackage{lipsum}
\usepackage{amsfonts}
\usepackage{graphicx}
\usepackage{epstopdf}
\usepackage{algorithmic}
\usepackage{hyperref}
\ifpdf
  \DeclareGraphicsExtensions{.eps,.pdf,.png,.jpg}
\else
  \DeclareGraphicsExtensions{.eps}
\fi

\usepackage{enumitem}
\setlist[enumerate]{leftmargin=.5in}
\setlist[itemize]{leftmargin=.5in}

\newsiamremark{remark}{Remark}
\newsiamremark{hypothesis}{Hypothesis}
\crefname{hypothesis}{Hypothesis}{Hypotheses}
\newsiamthm{claim}{Claim}

\headers{Kolmogorov-Arnold Networks for Dynamics (KANDy)}{K. Slote, J. Fish, and E. Bollt}

\title{KANDy: Kolmogorov-Arnold Networks and Dynamical System Discovery}

\author{
Kevin Slote%
\thanks{Clarkson University Center for Complex Systems Science, Clarkson University,
Potsdam, NY 13676
(\email{kslote@clarkson.edu}, \email{kslote1@gmail.com})}%
\thanks{Personal webpage: \url{http://www.kevin-slote.com}}%
\and
Jeremie Fish%
\footnotemark[1]%
\and
Erik Bollt\footnotemark[1]\thanks{Erik Bollt passed away during the preparation of this manuscript. This work reflects his intellectual contributions and guidance prior to his passing. We dedicate this paper to his memory.}%
}

\usepackage{amsopn}

\makeatletter
\newcommand*{\addFileDependency}[1]{
  \typeout{(#1)}
  \@addtofilelist{#1}
  \IfFileExists{#1}{}{\typeout{No file #1.}}
}
\makeatother



%% file: ex_supplement.tex







\section*{Ikeda Optical Cavity Map}

\begin{figure}
    \centering
    \includegraphics[width=1.0\linewidth]{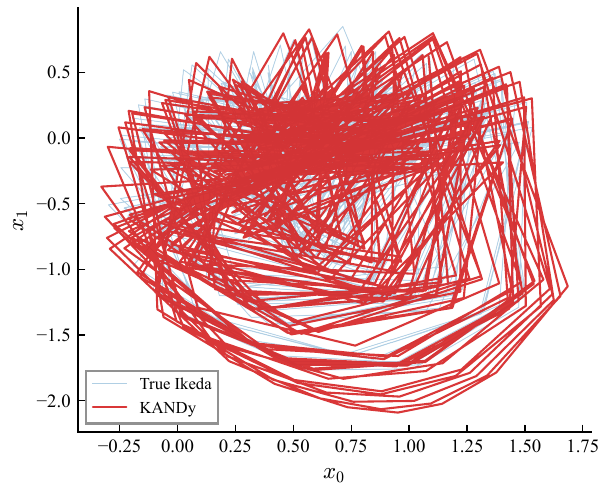}
    \caption{The learned Ikeda optical cavity map learned by KANDy (red) overlayed with the true attractor (grey).}
    \label{fig:ikeda_attractor}
\end{figure}

\begin{figure}
    \centering
    \includegraphics[width=1.0\linewidth]{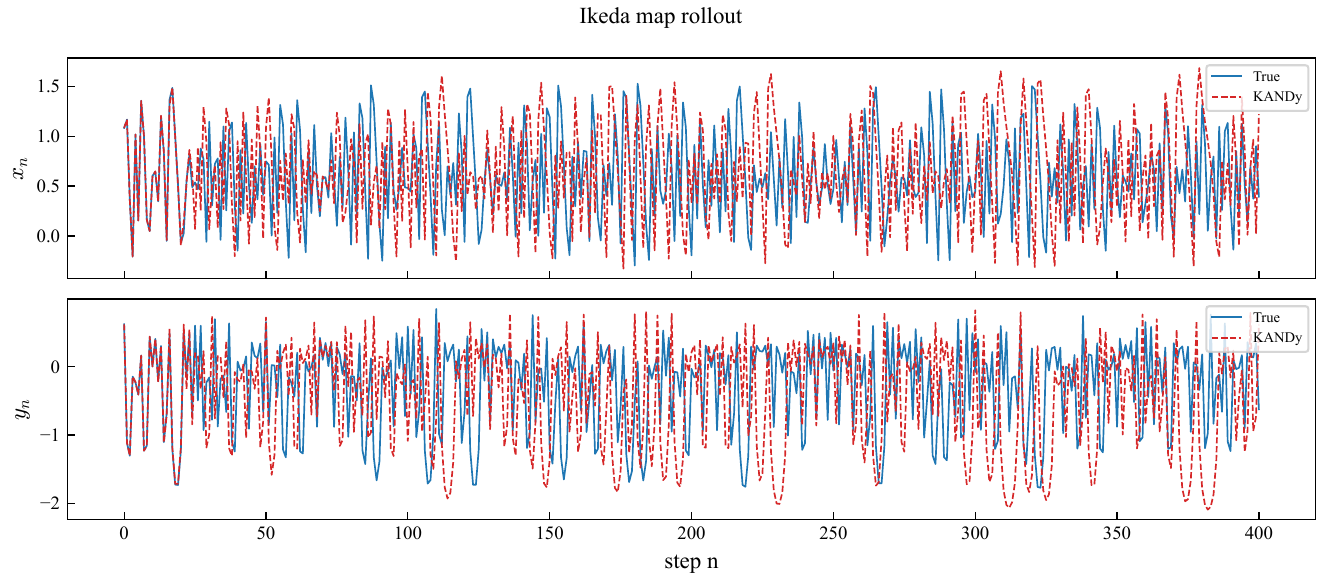}
    \caption{True Ikeda time series (blue) with the KANDy learned time series (red) on top.}
    \label{fig::ikea_timeseries}
\end{figure}

Figure~\ref{fig:ikeda_attractor} shows the learned vector field that KANDy approximated with RMSE $0.88$ on a rollout of 400 steps and a one-step MSE of $4\times 10^{-8}$. The lifted features included a pre-computation of $6/(1+r^2) - 0.4$. The lifted feaures were $\phi(x,y) = [ux\cos(t), uy\cos(t), ux\sin(t), uy\sin(t)]$, the netwok was $[4,2]$ with one input corresponding to one component. The training occurred in two phases: phase A performed 200 steps of LBFGS on a one-step rollout, and phase 2 performed rollout fine-tuning with rollout weight$=0.2$ and learning rate$=0.001$. The integration of the derivative loss was performed with the ``increment trick," taking the computed lifted dynamics and subtracting $map(s) - s$, e.g., Euler with $dt=1$. Symbolic extraction was performed with a library with down-weighted complexity terms for trigonometric functions, called with an $R^2$ threshold of $0.80$.

\begin{table}[h!]
\centering
\begin{tabular}{|l|l|l|}
\hline
\textbf{Term} & \textbf{Recovered Value} & \textbf{True}\\ \hline
 $x\cos(t)$ in $x_{n+1}$ & $0.9000$ & $0.9$ \\ \hline
$y\sin(t)$ in $x_{n+1}$ & $-0.9000$ & $-0.9$ \\ \hline
Constant term in $x_{n+1}$ & $1.0001$ & $1.0$ \\ \hline
\end{tabular}
\caption{The coefficients and errors for the Ikeda optical cavity map.}
\label{tab:ikeda_coeffs}
\end{table}

Table~\ref{tab:ikeda_coeffs} shows the estimated coefficients, which are remarkably close to the true system, making the learn equation.

\begin{align*}
x_{n+1} &= 1.0001 + 0.9x\cos(t) - 0.9y\sin(t), \\
y_{n+1} &= 0.9x\sin(t) + 0.9y\cos(t).
\end{align*}